\documentclass[11pt]{amsart}
\usepackage{amsmath, amssymb, amsfonts, verbatim, graphicx, amsthm, tikz-cd, mathrsfs, color, comment, fancyhdr, amsrefs}
\usepackage[top=1in, bottom=1in, left=1in, right=1in]{geometry}
\usepackage{enumerate}

\usepackage{mathtools}

\usepackage{url}

\usepackage{enumitem}

\usepackage[symbol]{footmisc}

\usepackage[foot]{amsaddr}

\allowdisplaybreaks[1]

\usepackage[sc]{mathpazo}
\usepackage[T1]{fontenc}

\newcommand{\vertiii}[1]{{\left\vert\kern-0.25ex\left\vert\kern-0.25ex\left\vert #1
		\right\vert\kern-0.25ex\right\vert\kern-0.25ex\right\vert}}

\newcommand{\Folner}{Følner}

\theoremstyle{plain}
\newtheorem{Thm}{Theorem}[section]
\newtheorem{Prop}[Thm]{Proposition}
\newtheorem{Lem}[Thm]{Lemma}
\newtheorem{Cor}[Thm]{Corollary}

\theoremstyle{definition}
\newtheorem{Def}[Thm]{Definition}
\newtheorem{Not}[Thm]{Notation}
\newtheorem{Rmk}[Thm]{Remark}
\newtheorem{Ex}[Thm]{Example}

\renewcommand{\epsilon}{\varepsilon}

\usepackage{setspace}

\usepackage{hyperref}

\title{Temporo-spatial differentiations with respect to finite unions of balls}

\author{Aidan Young$^1$}
\address{University of North Carolina at Chapel Hill}

\email{$^1$\url{aidanjy@live.unc.edu}}
\begin{document}
	
	\maketitle
	
	\begin{abstract}
		Here we study temporo-spatial differentiation problems with respect to sequences of finite unions of balls. We establish several convergence results, as well as construct pathological temporo-spatial differentiations with prescribed sets of limit points. We also demonstrate the prevalence of certain pathological temporo-spatial differentiations in the presence of a specification-like property.
	\end{abstract}

Temporo-spatial differentiations were introduced in \cite{Assani-Young} under the name of spatial-temporal differentiations. The central temporo-spatial differentiation problem is this: Given a continuous action $T : G \curvearrowright X$ of a discrete semigroup $G$ on a compact metric space $X$ with a sequence $(F_k)_{k = 1}^\infty$ of nonempty finite subsets of $G$, a Borel probability measure $\mu$ on $X$, a bounded measurable function $f : X \to \mathbb{C}$, and a sequence $(C_k)_{k = 1}^\infty$ of measurable subsets of $X$ with positive measure, what can be said about the limiting behavior of the sequence
$$
\left( \frac{1}{\mu(C_k)} \int_{C_k} \frac{1}{|F_k|} \sum_{j \in F_k} T_j f \mathrm{d} \mu \right)_{k = 1}^\infty ?
$$
In this article, we focus on the case where the spatial averaging sequence $(C_k)_{k = 1}^\infty$ consists of finite unions of balls, a setting we call ``multi-local." We study sufficient conditions for these corresponding temporo-spatial differentiations to converge, as well as the existence and prevalence of pathological multi-local temporo-spatial differentiations.
	
	In Section \ref{Notation}, we establish several notations that will be used throughout the article, as well as some standing assumptions and conventions.
	
	In Section \ref{Convergence}, we provide sufficient conditions for multi-local temporo-spatial differentations to converge. We also show how these convergence results can fail if certain assumptions are relaxed.
	
	In Section \ref{EO}, we briefly present the theory of ergodic optimization. In particularly, we characterize the maximum ergodic average in the context of continuous actions of amenable groups.
	
	In Section \ref{Phasing problems}, we construct multi-local temporo-spatial differentiations for a given real-valued continuous function $f$ which have a prescribed compact set $\mathcal{K}$ as the set of limit points of the temporo-spatial differentiation.
	
	In Section \ref{Chasing measures}, we consider temporo-spatial differentiations as sequences of measures
	$$\left( f \mapsto \frac{1}{\mu(C_k)} \int_{\mu(C_k)} \sum_{j = 0}^{k - 1} T^j f \mathrm{d} \mu \right)_{k = 1}^\infty ,$$
	and consider how to construct sequences $(C_k)_{k = 1}^\infty$ for which \linebreak$\operatorname{LS} \left( \left( f \mapsto \frac{1}{\mu(C_k)} \int_{C_k} \sum_{j = 0}^{k - 1} T^j f \mathrm{d} \mu \right)_{k = 1}^\infty \right)$ is some prescribed subset of the Choquet simplex of $T$-invariant Borel probability measures on $X$, where $\operatorname{LS} \left( \left( z_k \right)_{k = 1}^\infty \right)$ denotes the set of all limits of convergent subsequences of $(z_k)_{k = 1}^\infty$ (defined in more detail in Section \ref{Notation}). In particular, we construct examples of $(C_k)_{k = 1}^\infty$ for which $\operatorname{LS} \left( \left( f \mapsto \frac{1}{\mu(C_k)} \int_{C_k} \sum_{j = 0}^{k - 1} T^j f \mathrm{d} \mu \right)_{k = 1}^\infty \right)$ is the entire Choquet simplex of $T$-invariant measures.
	
	In Section \ref{Specification}, we show that for a system $(X, T)$ with a specification-like property that we call the Very Weak Specification Property, there exists a residual set of $x \in X$ that exhibit a strong form of the maximal Birkhoff average oscillation property. Specifically, there exists a residual set of $x \in X$ such that $\operatorname{LS} \left( \left( \mu_{x, \pi(k)} \right)_{k = 1}^\infty \right)$ is the entire Choquet simplex of $T$-invariant measures for all non-constant polynomials $\pi(t) \in \mathbb{Q}[t]$ such that $\pi(\mathbb{N}) \subseteq \mathbb{N}$, where $\mu_{x, k}$ are the the empirical measures $\mu_{x, k} = \frac{1}{k} \sum_{j = 0}^{k - 1} \delta_x \circ T^j$ for $x \in X$. Consequently, for sequences $(r_k)_{k = 1}^\infty$ of radii decaying to $0$ sufficiently fast, we have for a residual set of $x \in X$ that $\operatorname{LS} \left( \left( f \mapsto \frac{1}{\mu(B(x; r_k))} \int_{B(x; r_k)} \sum_{j = 0}^{\pi(k) - 1} T^j f \mathrm{d} \mu \right)_{k = 1}^\infty \right) $ is the entire Choquet simplex of $T$-invariant measures for all non-constant integer-valued polynomials $\pi(t)$ that send nonnegative integers to nonnegative integers.
	
	\section{Notations and conventions}\label{Notation}
	
	Here we identify particular notations and conventions we adopt throughout this article. Individual sections might place additional assumptions on some of the objects we define here. We also place more novel definitions in the later sections of the article.
	
	We will let $(X, \rho)$ be a compact metric space, and $T : G \curvearrowright X$ will be a continuous monoidal left-action of a discrete monoid $G$ on $X$ by continuous maps $(T_g)_{g \in G}$ (not necessarily invertible). That is to say, the maps $(T_g)_{g \in G}$ will satisfy the laws
	\begin{align*}
		T_{g_1} \circ T_{g_2}	& = T_{g_1 g_2}	& (\forall g_1, g_2 \in G) , \\
		T_{1_G}	& = \operatorname{id}_X ,
	\end{align*}
	where $1_G$ denotes the identity element of $G$. We will use $\mu$ to denote a Borel probability measure on $X$, though we will not in general assume that $\mu$ is $T$-invariant. The support of $\mu$ will be denoted $\operatorname{supp}(\mu)$.
	
	Given a finite subset $F$ of $G$, and a function $f : X \to \mathbb{C}$, we write
	$$
	\operatorname{Avg}_F f : = \frac{1}{|F|} \sum_{g \in F} T_g f ,
	$$
	where $T_g f : = f \circ T_g$. Similarly, if $\beta$ is a Borel probability measure on $X$, and $E \subseteq X$ is a Borel subset of $X$, we will write
	$$
	\left( \beta \circ \operatorname{Avg}_F \right) (E) : = \frac{1}{|F|} \sum_{g \in F} \beta \left( T_g^{-1} E \right) .
	$$
	These notations are consistent with each other in the sense that if $f \in C(X)$, then
	$$
	\int_X f \mathrm{d} \left( \beta \circ \operatorname{Avg}_F \right) = \int_X \operatorname{Avg}_F f \mathrm{d} \beta .
	$$
	If no domain is specified for an integral $\int$, then the integral is assumed to be over $X$, i.e. $\int : = \int_X$.
	
	We will denote the space of all Borel probability measures on $X$ by $\mathcal{M}(X)$. We will always consider $\mathcal{M}(X)$ with the weak*-topology, making $\mathcal{M}_T(X)$ a Choquet simplex. We use $\mathcal{M}_T(X)$ to denote the space of $T$-invariant Borel probability measures on $X$, also equipped with the weak*-topology to make $\mathcal{M}_T(X)$ a Choquet simplex.
	
	We use $\partial_e S$ to denote the set of extreme points of a subset $S$ of a real topological vector space, i.e. $\partial_e S$ denotes the set of all points in $S$ which cannot be expressed nontrivially as a convex combination of points in $S$.
	
	We will use $\mathbb{N}$ to denote the set of positive integers, and $\mathbb{N}_0$ to denote the set of nonnegative integers.
	
	A sequence $(F_k)_{k = 1}^\infty$ of finite subsets of a group $G$ is called \emph{\Folner} if
	\begin{align*}
		\lim_{k \to \infty} \frac{\left| F_k \Delta g F_k \right|}{|F_k|}	& = 0	& (\forall g \in G) ,
	\end{align*}
	where $|\cdot|$ denotes cardinality and $\Delta$ is the symmetric difference, i.e. $A \Delta B = (A \setminus B) \cup (B \setminus A)$.

	Given a sequence $(z_k)_{k = 1}^\infty$ in a topological space $Z$, we write
	$$\operatorname{LS} \left( (z_k)_{k = 1}^\infty \right) : = \left\{ \lim_{\ell \to \infty} z_{k_\ell} : k_1 < k_2 < \cdots, \; \textrm{$\lim_{\ell \to \infty} z_{k_\ell}$ exists} \right\}$$
	to denote the set of limit points of $(z_k)_{k = 1}^\infty$, called the \emph{limit set} of $(z_k)_{k = 1}^\infty$.
	
	\section{Convergence results and their limitations}\label{Convergence}
	
	\begin{Def}
		Let $\bar{x} = \left( x^{(1)}, \ldots, x^{(n)} \right) \in X^n, \bar{r} = \left( r^{(1)} , \ldots, r^{(n)} \right) \in (0, \infty)^n$. We write
		$$B \left( \bar{x} ; \bar{r} \right) : = \bigcup_{h = 1}^n B\left( x^{(h)}, r^{(h)} \right) , $$
		where $B \left( x ; r \right) : = \left\{ y \in X : \rho(x, y) < r \right\}$ is the open ball with center $x$ and radius $r$. We refer to sets of the form $B(\bar{x} ; \bar{r})$ as \emph{multi-balls}.
	\end{Def}
	
	\begin{Lem}\label{Multi-ball with distinct centers}
		Every multi-ball $B \left( x^{(1)}, \ldots, x^{(n)} ; r^{(1)}, \ldots, r^{(n)} \right)$ can be expressed in the form $$B \left( y^{(1)}, \ldots, y^{(m)} ; s^{(1)}, \ldots, s^{(m)} \right) ,$$
		where $y^{(1)}, \ldots, y^{(m)}$ are distinct.
	\end{Lem}
	
	\begin{proof}
		If $x^{(1)}, \ldots, x^{(n)}$ are not already distinct, then we can write \linebreak$\left\{ x^{(1)}, \ldots, x^{(n)} \right\} = \left\{ x^{(h_1)}, \ldots, x^{(h_m)} \right\}$, where $h_1, \ldots, h_m \in \{1, \ldots, n\}$, and $x^{(h_1)}, \ldots, x^{(h_m)}$ are distinct. Then
		$$B \left( x^{(1)}, \ldots, x^{(n)} ; r^{(1)}, \ldots, r^{(n)} \right) = B \left( y^{(1)}, \ldots, y^{(p)} ; s^{(1)}, \ldots, s^{(m)} \right) ,$$
		where $y^{(p)} = x^{(h_p)} , s^{(p)} = \max \left\{ r^{(h)} : x^{(h)} = y^{(p)} \right\}$.
	\end{proof}
	
	\begin{Def}\label{Rapid decay}
		Let $(X, \rho)$ be a compact metric space, and let $T : G \curvearrowright X$ be an action of a discrete semigroup $G$ by H\"older maps $T_g$ equipped with functions $H, L : G \to (0, \infty)$ such that
		\begin{align*}
			\rho \left( T_g x, T_g y \right)	& \leq L(g) \cdot \rho(x, y)^{H(g)}	& (\forall g \in G, x \in X, y \in X) .
		\end{align*}
		We refer to the pair $(H, L)$ as a \emph{modulus of H\"older continuity} (abbreviated \emph{MoH\"oC}) for $T$. Let $\mathbf{F} = (F_k)_{k = 1}^\infty$ be a sequence of nonempty finite subsets of $G$. We say that a sequence $\left( \bar{r}_k \right)_{k = 1}^\infty$ of $n$-tuples $\bar{r}_k = \left( r_k^{(1)}, \ldots, r_k^{(n)} \right)_{k = 1}^\infty$ of positive numbers \emph{decays $(X, \rho, H, L, \mathbf{F})$-fast} if
		\begin{align*}
			\lim_{k \to \infty} \frac{ \left| \left\{ g \in F_k : L(g) \cdot \left( r_k^{(h)} \right)^{H(g)} > \delta \right\} \right| }{|F_k|}	& = 0	& (\forall \delta \in (0, \infty), h \in \{1, \ldots, n\}) , \\
			\lim_{k \to \infty} r_k^{(h)}	& = 0	& (\forall h \in \{1, \ldots, n\})  .
		\end{align*}
	\end{Def}
	
	An immediate observation about this definition is that if $\left( \bar{r}_k \right)_{k = 1}^\infty$ is a sequence of $n$-tuples of positive numbers that decay $(X, \rho, H, L, \mathbf{F})$-fast, and $\left( \bar{s}_k \right)_{k = 1}^\infty$ is another sequence of $n$-tuples of positive numbers for which there exists $K \in \mathbb{N}$ such that $s_k^{(h)} \leq r_k^{(h)}$ for all $h \in \{1, \ldots, n\}, k \geq K$, then $\left( \bar{s}_k \right)_{k = 1}^\infty$ decays $(X, \rho, H, L, \mathbf{F})$-fast. So we have in fact described a rapid decay condition. Moreover, any system of H\"older maps with MoH\"oC $(H, L)$ will admit a sequence $(r_k)_{k = 1}^\infty$ that decays $(X, \rho, H, L, \mathbf{F})$-fast.
	
	Our assumption that $\lim_{k \to \infty} r_k^{(h)} = 0$ ensures that if $x^{(1)}, \ldots, x^{(h)}$ are distinct points in $X$, then the balls $\left\{ B \left( x^{(h)} ; r_k^{(h)} \right) \right\}_{h = 1}^n$ are pairwise disjoint for sufficiently large $k$, since for sufficiently large $k$ we'll have that
	$$\max \left\{ r_k^{(1)}, \ldots, r_k^{(n)} \right\} < \frac{1}{2} \min \left\{ \rho \left( x^{(h_1)}, x^{(h_2)} \right) : 1 \leq h_1 < h_2 \leq n \right\} .$$
	
	For the remainder of this section, $T : G \curvearrowright X$ will be an action of a discrete group $G$ on $X$ by H\"older homoeomorphisms with MoH\"oC $(H, L)$, and $\mathbf{F} = (F_k)_{k = 1}^\infty$ will be a sequence of nonempty finite subsets of $G$.
	
	\begin{Not}
	Let $\mu$ be a Borel probability measure on $X$, and let $E \subseteq X$ be a $\mu$-measurable set such that $\mu(E) > 0$. The functional $\alpha_E : C(X) \to \mathbb{C}$ is defined as
	$$\alpha_E(f) : = \frac{1}{\mu(E)} \int_E f \mathrm{d} \mu .$$
	We will sometimes also treat $\alpha_E$ instead as a Borel probability measure $\alpha_E : A \mapsto \mu(A \vert E)$. These interpretations are consistent with each other in the sense that $\alpha_E(f) = \int f \mathrm{d} \alpha_E$ for all $f \in C(X)$.
	\end{Not}
	
	\begin{Lem}\label{Singly local pointwise reduction}
		Let $x \in X$, and let $(r_k)_{k = 1}^\infty$ be a sequence of positive numbers that decays \linebreak$(X, \rho, H, L, \mathbf{F})$-fast, and suppose $f \in C(X)$. Let $\mu$ be a Borel probability measure on $X$, and let $x \in \operatorname{supp}(\mu)$. Then
		$$\lim_{k \to \infty} \left( \alpha_{B(x; r_k)} \left( \operatorname{Avg}_{F_k} f \right) - \operatorname{Avg}_{F_k} f(x) \right) = 0 .$$
		Moreover, if $f$ satisfies the H\"older condition
		\begin{align*}
			|f(y) - f(z)|	& \leq c \cdot \rho(y, z)^\beta	& (\forall y, z \in X) ,
		\end{align*}
		for some constants $c, \beta \in (0, \infty)$, then
		$$\left| \alpha_{B(x; r_k)} \left( \operatorname{Avg}_{F_k} f \right) - \operatorname{Avg}_{F_k} f(x) \right| \leq \frac{c}{|F_k|} \sum_{g \in F_k} L(g)^\beta \cdot r_k^{\beta H(g)} .$$
	\end{Lem}
	
	\begin{proof}
		Fix $\epsilon > 0$. Since $f$ is continuous and $X$ is compact, we know that $f$ is \emph{uniformly} continuous, meaning that there exists $\delta > 0$ such that $$\rho(y, z) \leq \delta \Rightarrow |f(y) - f(z)| \leq \epsilon .$$
		Set
		$$A_k = \left\{ g \in F_k : L(g) \cdot r_k^{H(g)} \geq \delta \right\} .$$
		By the hypothesis that $(r_k)_{k = 1}^\infty$ decays $(X, \rho, H, L, \mathbf{F})$-fast, we know that $\lim_{k \to \infty} \frac{|A_k|}{|F_k|} = 0$.
		
		We estimate
		\begin{align*}
			& \left|\alpha_{B(x; r_k)} \left( \operatorname{Avg}_{F_k} f \right) - \operatorname{Avg}_{F_k} f(x) \right| \\
			=	& \left| \frac{1}{\mu(B(x; r_k))} \int_{B(x; r_k)} \frac{1}{|F_k|} \sum_{g \in F_k} \left( f(T_g y) - f(T_g x) \right) \mathrm{d} \mu(y) \right| \\
			\leq	& \frac{1}{|F_k|} \sum_{g \in F_k} \frac{1}{\mu(B(x; r_k))} \int_{B(x; r_k)} \left| f(T_g y) - f(T_g x) \right| \mathrm{d} \mu(y)	& (\dagger) \\
			=	& \left( \frac{1}{|F_k|} \sum_{g \in A_k} \frac{1}{\mu(B(x; r_k))} \int_{B(x; r_k)} \left| f(T_g y) - f(T_g x) \right|  \mathrm{d} \mu(y) \right) \\
			& + \left( \frac{1}{|F_k|} \sum_{g \in F_k \setminus A_k} \frac{1}{\mu(B(x; r_k))} \int_{B(x; r_k)} \left| f(T_g y) - f(T_g x) \right| \mathrm{d} \mu(y) \right) .
		\end{align*}
		We will return to the line marked $(\dagger)$ when we compute the estimate for the case where $f$ is H\"older. For now, we estimate these two terms separately.
		\begin{align*}
				& \frac{1}{|F_k|} \sum_{g \in A_k} \frac{1}{\mu(B(x; r_k))} \int_{B(x; r_k)} \left| f(T_g y) - f(T_g x) \right| \mathrm{d} \mu(y)	\\
			\leq	& \frac{1}{|F_k|} \sum_{g \in A_k} \frac{1}{\mu(B(x; r_k))} \int_{B(x; r_k)} \left\| 2 f \right\|_{C(X)} \mathrm{d} \mu(y) \\
			=	& \frac{2 |A_k|}{|F_k|} \|f\|_{C(X)} .
		\end{align*}
		Choose $K \in \mathbb{N}$ sufficiently large that $\frac{|A_k|}{|F_k|} \leq \epsilon$. Then for $k \geq K$, we have that
		$$\frac{2 |A_k|}{|F_k|} \|f\|_{C(X)} \leq 2 \|f\|_{C(X)} \epsilon .$$
		
		For the other of second of the two aforementioned terms, we observe that if $g \in F_k \setminus A_k$, then
		\begin{align*}
			\rho(T_g y, T_g x)	& \leq L(g) \cdot \rho(x, y)^{H(g)} \\
			& \leq \delta \\
			\Rightarrow |f(T_g y) - f(T_g x)|	& \leq \epsilon .
		\end{align*}
		Thus
		\begin{align*}
				& \frac{1}{|F_k|} \sum_{g \in F_k \setminus A_k} \frac{1}{\mu(B(x; r_k))} \int_{B(x; r_k)} \left| f(T_g y) - f(T_g x) \right| \mathrm{d} \mu(y) \\
			\leq	& \frac{1}{|F_k|} \sum_{g \in F_k \setminus A_k} \frac{1}{\mu(B(x; r_k))} \int_{B(x; r_k)} \epsilon \mathrm{d} \mu(y) \\
			=	&\frac{|F_k| - |A_k|}{|F_k|} \epsilon \\
			\leq	& \epsilon .
		\end{align*}
		
		Therefore, if $k \geq K$, then
		$$
		\left|\alpha_{B(x; r_k)} \left( \operatorname{Avg}_{F_k} f \right) - \operatorname{Avg}_{F_k} f(x) \right| \leq \left( 2 \|f\|_{C(X)} + 1 \right) \epsilon .
		$$
		
		Finally, in the case where we have the additional hypothesis that $f$ is $(c, \beta)$-H\"older, we can instead estimate the earlier $(\dagger)$ as
		\begin{align*}
			& \frac{1}{|F_k|} \sum_{g \in F_k} \frac{1}{\mu(B(x; r_k))} \int_{B(x; r_k)} \left| f(T_g y) - f(T_g x) \right| \mathrm{d} \mu(y) \\
			\leq	& \frac{1}{|F_k|} \sum_{g \in F_k} \frac{1}{\mu(B(x; r_k))} \int_{B(x; r_k)} c \cdot \rho(T_g y, T_g x)^\beta \mathrm{d} \mu(y) \\
			\leq	& \frac{1}{|F_k|} \sum_{g \in F_k} \frac{1}{\mu(B(x; r_k))} \int_{B(x; r_k)} c \cdot \left( L(g) \cdot \rho(x, y)^{H(g)} \right)^\beta \mathrm{d} \mu(y) \\
			\leq	& \frac{1}{|F_k|} \sum_{g \in F_k} \frac{1}{\mu(B(x; r_k))} \int_{B(x; r_k)} c \cdot \left( L(g) \cdot r_k^{H(g)} \right)^\beta \mathrm{d} \mu(y) \\
			=	& \frac{1}{|F_k|} \sum_{g \in F_k} \frac{1}{\mu(B(x; r_k))} \int_{B(x; r_k)} c \cdot \left( L(g)^\beta \cdot r_k^{\beta H(g)} \right) \mathrm{d} \mu(y) \\
			=	& \frac{c}{|F_k|} \sum_{g \in F_k} L(g)^\beta \cdot r_k^{\beta H(g)} .
		\end{align*}
	\end{proof}
	
	The upshot of Lemma \ref{Singly local pointwise reduction} is that when we consider temepero-spatial differentiations with respect to balls of radius decaying sufficiently fast centered at a fixed point $x_0$, this temporo-spatial differentiation is equivalent to a pointwise (temporal) ergodic average. On one hand, this means that we can consider certain ``random" temepero-spatial differentiations by appealing to pointwise convergence theorems, as in Corollary \ref{Random TSD's, Lindenstrauss}. On another hand, this means that we can use pathological pointwise ergodic averages to generate pathological temporo-spatial differentiations, as we will see in Section \ref{Specification}.
	
	The following lemma lets us describe temporo-spatial averages over multi-balls in terms of temporo-spatial averages over balls, and will be useful going forward.
	
	\begin{Lem}\label{Phasing for multi-balls}
		Let $\mu$ be a Borel probability measure on $X$, and let \linebreak$x^{(1)}, \ldots, x^{(n)} \in \operatorname{supp}(\mu) ; r^{(1)}, \ldots, r^{(n)} \in (0, 1)$ such that the balls $\left\{ B\left( x^{(h)} ; r^{(h)} \right) \right\}_{h = 1}^n$ are pairwise disjoint. Let $f \in L^1(X, \mu)$. Then
		$$
		\alpha_{B(\bar{x}, \bar{r})}(f) = \sum_{h = 1}^n \frac{\mu\left( B \left( x^{(h)} ; r^{(h)} \right) \right)}{\mu\left( B \left( x^{(1)} ; r^{(1)} \right) \right) + \cdots + \mu\left( B \left( x^{(n)} ; r^{(n)} \right) \right)} \alpha_{B \left( x^{(h)} ; r^{(h)} \right)} (f) .
		$$
	\end{Lem}
	
	\begin{proof}
		\begin{align*}
			\alpha_{B(\bar{x};, \bar{r})}(f)	& = \frac{1}{\mu(B(\bar{x};, \bar{r}))} \int_{B(\bar{x}; \bar{r})} f \mathrm{d} \mu \\
			& = \sum_{h = 1}^n \frac{1}{\sum_{u = 1}^n \mu \left( B \left( x^{(u)} ; r^{(u)} \right) \right)} \int_{B \left( x^{(h)}; r^{(h)} \right) } f \mathrm{d} \mu \\
			& = \sum_{h = 1}^n \frac{\mu\left( B \left( x^{(h)} ; r^{(h)} \right) \right)}{\sum_{u = 1}^n \mu \left( B \left( x^{(u)} ; r^{(u)} \right) \right)} \frac{1}{\mu \left( B \left( x^{(h)} ; r^{(h)} \right) \right)} \int_{B \left( x^{(h)}; r^{(h)} \right) } f \mathrm{d} \mu \\
			& = \sum_{h = 1}^n \frac{\mu\left( B \left( x^{(h)} ; r^{(h)} \right) \right)}{\mu\left( B \left( x^{(1)} ; r^{(1)} \right) \right) + \cdots + \mu\left( B \left( x^{(n)} ; r^{(n)} \right) \right)} \alpha_{B \left( x^{(h)} ; r^{(h)} \right)} (f)
		\end{align*}
	\end{proof}
	
	\begin{Thm}\label{Multi-local convergence}
		Let $\left( \bar{r}_k \right)_{k = 1}^\infty$ be a sequence that decays $(X, \rho, H, L, \mathbf{F})$-fast, and let $f \in C(X)$. Suppose $\bar{x} = \left( x^{(1)}, \ldots, x^{(n)} \right)$ is an $n$-tuple in $X$ such that
		\begin{align*}
			\lim_{k \to \infty} \operatorname{Avg}_{F_k} f\left( x^{(h)} \right)	& = C	& (\forall h \in \{1, \ldots, n\}) ,
		\end{align*}
		where $C$ is independent of $h$, and let $\mu$ be a Borel probability measure on $X$ for which $x^{(1)}, \ldots, x^{(n)} \in \operatorname{supp}(\mu)$. Then
		$$
		\lim_{k \to \infty} \alpha_{B (\bar{x} ; \bar{r}_k)} \left( \operatorname{Avg}_{F_k} f \right) = C .
		$$
	\end{Thm}
	
	\begin{proof}
		By Lemma \ref{Multi-ball with distinct centers}, we can assume without loss of generality that $x^{(1)}, \ldots, x^{(n)}$ are distinct. Because $r_k^{(h)} \to 0$, we know that for sufficiently large $k$, we'll have
		$$B \left( \bar{x} ; \bar{r}_k \right) = \bigsqcup_{h = 1}^n B\left( x^{(h)} ; r_k^{(h)} \right) , $$
		where $B \left( x ; r \right) : = \left\{ y \in X : \rho(x, y) < r \right\}$ is the open ball with center $x$ and radius $r$, and $\sqcup$ denotes disjoint union. We therefore estimate that
		\begin{align*}
			& \left| \alpha_{B(\bar{x} ; \bar{r}_k)} \left( \operatorname{Avg}_{F_k} f \right) - C \right| \\
			=	& \left| \frac{1}{\mu(B(\bar{x} ; \bar{r}_k))} \int_{B(\bar{x} ; \bar{r}_k)} \frac{1}{|F_k|} \sum_{g \in F_k} \left( f(T_g y) - C \right) \mathrm{d} \mu(y) \right| \\
			=	& \left| \frac{1}{\mu(B(\bar{x} ; \bar{r}_k))} \sum_{h = 1}^n \int_{B\left( x^{(h)} ; r_k^{(h)} \right)} \frac{1}{|F_k|} \sum_{g \in F_k} \left( f(T_g y) - C \right) \mathrm{d} \mu(y) \right| \\
			\leq	& \sum_{h = 1}^n \left| \frac{1}{\mu(B(\bar{x}; \bar{r}_k))} \int_{B\left( x^{(h)}; r_k^{(h)} \right)} \frac{1}{|F_k|} \sum_{g \in F_k} \left( f(T_g y) - C \right) \mathrm{d} \mu(y) \right| \\
			\leq	& \sum_{h = 1}^n \left| \frac{1}{\mu\left(B\left(x^{(h)}; r_k^{(h)}\right) \right)} \int_{B\left( x^{(h)}; r_k^{(h)} \right)} \frac{1}{|F_k|} \sum_{g \in F_k} \left( f(T_g y) - C \right) \mathrm{d} \mu(y) \right| \\
			=	& \sum_{h = 1}^n \left| \alpha_{B \left( x^{(h)}; r_k^{(h)} \right)} \left( \operatorname{Avg}_{F_k} f - C \right) \right| \\
			\leq	& \sum_{h = 1}^n \left( \left| \alpha_{B \left( x^{(h)}; r_k^{(h)} \right)} \left( \operatorname{Avg}_{F_k} f - \operatorname{Avg}_{F_k} (x) \right) \right| + \left| \alpha_{B \left( x^{(h)}; r_k^{(h)} \right)} \left( \operatorname{Avg}_{F_k} f(x) - C \right) \right| \right) \\
			=	& \sum_{h = 1}^n \left( \left| \alpha_{B \left( x^{(h)}; r_k^{(h)} \right)} \left( \operatorname{Avg}_{F_k} f - \operatorname{Avg}_{F_k} \left( x^{(h)} \right) \right) \right| + \left| \operatorname{Avg}_{F_k} f\left(x^{(h)}\right) - C \right| \right) \\
			\stackrel{k \to \infty}{\to}	& 0 ,
		\end{align*}
		where the limit in the last line follows from Lemma \ref{Singly local pointwise reduction}.
	\end{proof}
	
	We recall here the following definition.
	
	\begin{Def}
		Let $(F_k)_{k = 1}^\infty$ be a sequence of nonempty finite subsets of a group $G$. We say that $(F_k)_{k = 1}^\infty$ is \emph{tempered} if there exists a constant $c > 0$ such that
		\begin{align*}
			\left| \bigcup_{j = 1}^{k - 1} F_j^{-1} F_k \right|	& \leq c |F_k|	& (\forall k \geq 2) .
		\end{align*}
	\end{Def}
	
	\begin{Lem}\label{Tempered subsequence}
		Every \Folner \space sequence $(F_k)_{k = 1}^\infty$ has a tempered subsequence. In particular, every amenable group admits a tempered \Folner \space sequence.
	\end{Lem}
	
	\begin{proof}
		\cite[Proposition 1.4]{LindenstraussErgodicTheorem}
	\end{proof}
	
	The existence of tempered subsequences will be relevant to us in later sections.
	
	\begin{Cor}\label{Random TSD's, Lindenstrauss}
		Suppose $G$ is an amenable group, and $\mathbf{F}$ is a tempered \Folner \space sequence. Suppose further that $\mu$ is a Borel probability measure on $X$ that is $T$-invariant and ergdic. Then for almost all $\bar{x} \in X^n$, we have for all $f \in C(X)$ and all sequences $(\bar{r}_k)_{k = 1}^\infty$ that decay $(X, \rho, H, L, \mathbf{F})$-fast that
		$$
		\lim_{k \to \infty} \alpha_{B(\bar{x}; \bar{r}_k)} \left( \operatorname{Avg}_{F_k} f \right)	= \int f \mathrm{d} \mu .
		$$
	\end{Cor}
	
	\begin{proof}
		Since $X$ is compact metrizable, it follows that $C(X)$ is separable, so let $\{f_\ell\}_{\ell \in \mathbb{N}}$ be a countable dense subset of $C(X)$. For each $\ell \in \mathbb{N}$, set
		$$
		X_\ell = \left\{ x \in X : \operatorname{Avg}_{F_k} f_\ell(x) = \int f_\ell \mathrm{d} \mu \right\} .
		$$
		By the Lindenstrauss ergodic theorem \cite[Theorem 3.3]{LindenstraussErgodicTheorem}, each of these sets $X_\ell$ has full probability, and so $X' = \bigcap_{\ell \in \mathbb{N}} X_\ell$ also has full probability. Thus $\left( X' \right)^n$ is of full probability in $X^n$ with respect to the product measure $\underbrace{\mu \times \cdots \times \mu}_{n}$.
		
		Let $\bar{x} \in \left( X' \right)^n$, and let $(\bar{r}_k)_{k = 1}^\infty$ be a sequence of $n$-tuples of positive numbers that decay $(X, \rho, H, L, \mathbf{F})$-fast. By Theorem \ref{Multi-local convergence}, we know that $\lim_{k \to \infty} \alpha_{B(\bar{x}; \bar{r}_k)} \left( \operatorname{Avg}_{F_k} f_\ell \right) = \int f_\ell \mathrm{d} \mu$ for all $\ell \in \mathbb{N}$. Now it remains to prove that this convergence occurs for all $f \in C(X)$.
		
		Let $f \in C(X)$, and fix $\epsilon > 0$. Choose $f_\ell$ such that $\|f - f_\ell \|_{C(X)} \leq \epsilon$. Then
		\begin{align*}
			& \left| \int f \mathrm{d} \mu - \alpha_{B(\bar{x}; \bar{r}_k)} \left( \operatorname{Avg}_{F_k} f \right) \right| \\
			\leq	& \left| \int f \mathrm{d} \mu - \int f_\ell \mathrm{d} \mu \right| + \left| \int f_\ell \mathrm{d} \mu - \alpha_{B(\bar{x}; \bar{r}_k)} \left( \operatorname{Avg}_{F_k} f_\ell \right) \right| + \left| \alpha_{B(\bar{x}, \bar{r}_k)} \left( \operatorname{Avg}_{F_k} (f_\ell - f) \right) \right| \\
			\leq	& \left\| f - f_\ell \right\|_{C(X)} + \left| \int f_\ell \mathrm{d} \mu - \alpha_{B(\bar{x}; \bar{r}_k)} \left( \operatorname{Avg}_{F_k} f_\ell \right) \right| + \|f - f_\ell\|_{C(X)} \\
			\leq	& 2 \epsilon + \left| \int f_\ell \mathrm{d} \mu - \alpha_{B(\bar{x}; \bar{r}_k)} \left( \operatorname{Avg}_{F_k} f_\ell \right) \right| .
		\end{align*}
		Now choose $K \in \mathbb{N}$ such that if $k \geq K$, then $\left| \int f_\ell \mathrm{d} \mu - \alpha_{B(\bar{x}; \bar{r}_k)} \left( \operatorname{Avg}_{F_k} f_\ell \right) \right| \leq \epsilon$. Then for $k \geq K$, we have that
		$$\left| \int f \mathrm{d} \mu - \alpha_{B(\bar{x}; \bar{r}_k)} \left( \operatorname{Avg}_{F_k} f \right) \right| \leq 3 \epsilon .$$
		This demonstrates the convergence.
	\end{proof}
	
	Theorem \ref{Multi-local convergence} tells us that if we look at a sequence of concentric multiballs with rapidly vanishing radii, and if the pointwise Birkhoff averages at the centers converge to the same limit, then the temporo-spatial average with respect to these sequences of multiballs will inherit the limiting behavior f the pointwise Birkhoff averages. We might wonder whether Theorem \ref{Multi-local convergence} could be generalized by replacing the assumption that $\lim_{k \to \infty} \operatorname{Avg}_{F_k} f \left( x^{(h)} \right) = C$ with $\limsup_{k \to \infty} \operatorname{Avg}_{F_k} f \left( x^{(h)} \right) = C$, assuming of course that $f$ was real-valued. It turns out this generalization fails, as the next example demonstrates. 
	
	\begin{Ex}\label{Can't take lim sups}
		Let $X = \{0, 1\}^\mathbb{N}$, and let $\mu$ be the Borel probability measure on $X$ generated by
		$$\mu \left( [a_1, \ldots, a_\ell] \right) = 2^{-\ell}$$
		for all $a_1, \ldots, a_\ell \in \{0, 1\}, \ell \in \mathbb{N}$, where $[a_1, \ldots, a_\ell] = \left\{ x \in X : x(1) = a_1, \ldots, x(\ell) = a_\ell \right\}$. Let $T_j : \mathbb{N}_0 \curvearrowright X$ be the left shift $(Tx)(i) = x(i + j)$, where $\mathbb{N}_0$ denotes the semigroup of nonnegative integers, making $(X, \mu, T)$ a one-sided Bernoulli shift. Equip $X$ with the compatible metric
		$$\rho(x, y) = \begin{cases}
			0	& \textrm{if $x = y$}, \\
			2^{-\ell}	& \textrm{if $\ell = \min \{ i \in \mathbb{N} : x(i) \neq y(i) \}$} .
		\end{cases}$$
		Then $B\left(x; 2^{-k} \right) = [x(1), \ldots, x(k)]$, and $T_k$ is $2^{k}$-Lipschitz, i.e. $\rho \left( T_k x, T_k y \right) \leq 2^k \cdot \rho(x, y)$. Set $L(j) = 2^j, H(j) = 1$. We can check that $\left( 2^{-k} \right)_{k = 1}^\infty$ decays $(X, \rho, H, L, \mathbf{F})$-fast for $\mathbf{F} = (\{0,1 , \ldots, k - 1\})_{k = 1}^\infty$ by observing that for any $\delta > 0$, if $2^{j - k} \geq \delta$ for $\delta \in (0, 1), 0 \leq j \leq k - 1$, then $j - k \geq \log_2 \delta \iff j \geq k + \log_2 \delta$. Therefore $L(j) \cdot \left( 2^{-k} \right)^{H(j)} < \delta$ for all but at most $\lceil |\log_2 \delta| \rceil$ of $j \in \{0, 1, \ldots, k - 1\}$, so
		$$
		\frac{\left| \left\{ j \in F_k : L(j) \cdot \left( 2^{-k} \right)^{H(j)} \geq \delta \right\} \right|}{|F_k|} \leq \frac{\lceil |\log_2 \delta| \rceil}{k} \stackrel{k \to \infty}{\to} 0 .
		$$
		
		Let $(c_n)_{n = 1}^\infty$ be a sequence of natural numbers chosen to grow fast enough that
		\begin{align*}
			\frac{c_{n}}{c_1 + \cdots + c_n}	& \geq \frac{n - 1}{n}	& (\forall n \in \mathbb{N}) .
		\end{align*}
		Set $s_n = c_1 + \cdots + c_n$, so our growth condition states that $\frac{c_n}{s_n} \geq \frac{n - 1}{n}$. Now construct $x \in X$ by
		$$
		x(i) = \begin{cases}
			0	& 1 \leq i \leq s_1 \\
			1	& s_1 < i \leq s_2 , \\
			0	& s_2 < i \leq s_3 , \\
			\cdots \\
			0	& s_{2n} < i \leq s_{2n + 1} \\
			1	& s_{2n + 1} < i \leq s_{2n + 2} . \\
			\cdots
		\end{cases}
		$$
		In plain language, this $x$ consists of $c_1$ terms of $0$, then $c_2$ terms of $1$, then $c_3$ terms of $0$, then $c_4$ terms of $1$, etc. We then define $y \in X$ by
		\begin{align*}
			y(i)	& = 1 - x(i)	& (\forall i \in \mathbb{N}) ,
		\end{align*}
		i.e. replacing all $0$'s with $1$'s and vice-versa. Set $f = \chi_{[0]}$. We claim that $\limsup_{k \to \infty} \operatorname{Avg}_{F_k} f(x) = \limsup_{k \to \infty} \operatorname{Avg}_{F_k}(y) = 1$.
		
		Consider the case where we sample along $\left(s_{2n - 1} \right)_{n = 1}^\infty$. Then
		$$
		\operatorname{Avg}_{F_{s_{2n - 1}}} f(x) = \frac{c_1 + c_3 + c_5 + \cdots + c_{2n - 1}}{c_1 + c_2 + c_3 + c_4 + c_5 + c_6 + \cdots + c_{2n - 1}} \geq \frac{c_{2n - 1}}{s_{2n - 1}} \geq \frac{2n - 2}{2n - 1} \stackrel{n \to \infty}{\to} 1 .
		$$
		But $\operatorname{Avg}_{F_k} f(z) \in [0, 1]$ for all $k \in \mathbb{N}, z \in X$, so we can conclude that $\limsup_{k \to \infty} \operatorname{Avg}_{F_k} f(x) = 1$. Likewise, sampling along $s_{2n}$, we see that
		$$
		\operatorname{Avg}_{F_{s_{2n}}} f(y) = \frac{c_2 + c_4 + \cdots + c_{2n}}{s_{2n}} \geq \frac{c_{2n}}{s_{2n}} \geq \frac{2n - 1}{2n} \stackrel{n \to \infty}{\to} 1 .	
		$$
		Thus $\limsup_{k \to \infty} \operatorname{Avg}_{F_k}(y) = 1$.
		
		Computing the temporo-spatial averages, we can see that
		\begin{align*}
			& \alpha_{B\left(x, y; 2^{-k}, 2^{-k} \right)} \left( \operatorname{Avg}_{F_k} f \right) \\
			=	& \frac{1}{\mu \left( B\left(x, y; 2^{-k}, 2^{-k} \right) \right)} \int_{B\left(x, y; 2^{-k}, 2^{-k} \right)} \frac{1}{k} \sum_{j = 0}^{k - 1} \chi_{[0]}(T_j z) \mathrm{d} \mu(z) \\
			=	& \frac{1}{2^{-k} + 2^{-k}} \frac{1}{k} \sum_{j = 0}^{k - 1} \chi_{[0]}(T_j z) \mathrm{d} \mu(z) \\
			=	& 2^{k - 1} \int_{[x(1), \ldots, x(k)] \sqcup [y(1), \ldots, y(k)]} \frac{1}{k} \sum_{j = 0}^{k - 1} \chi_{[0]}(T_j z) \mathrm{d} \mu(z) \\
			= & 2^{k - 1} \int \frac{1}{k} \sum_{j = 0}^{k - 1} \left( \chi_{[x(1), \ldots, x(k)]}(z) + \chi_{[y(1), \ldots, y(k)]}(z) \right) \chi_{[0]}(T_j z) \mathrm{d} \mu(z) \\
			=	& 2^{k - 1} \int \frac{1}{k} \sum_{j = 0}^{k - 1} \left( \chi_{[x(1), \ldots, x(k)]}(z) + \chi_{[y(1), \ldots, y(k)]}(z) \right) \chi_{T^{-j} [0]}(z) \mathrm{d} \mu(z) \\
			=	& 2^{k - 1} \int \frac{1}{k} \sum_{j = 0}^{k - 1} \left( \chi_{[x(1), \ldots, x(k)] \cap T^{-j}[0]}(z) + \chi_{[y(1), \ldots, y(k)] \cap T^{-j}[0]}(z) \right)\mathrm{d} \mu(z)
		\end{align*}
		We know that
		\begin{align*}\chi_{[x(1), \ldots, x(k)] \cap T^{-j}[0]}(z)	& = \begin{cases}
				1	& \textrm{if } z(1 + j) = x(1 + j) = 0 , \\
				0	& \textrm{otherwise} ,
			\end{cases} \\
			\chi_{[y(1), \ldots, y(k)] \cap T^{-j}[0]}(z)	& = \begin{cases}
				1	& \textrm{if } z(1 + j) = y(1 + j) = 0 , \\
				0	& \textrm{otherwise} .
			\end{cases}
		\end{align*}
		Thus
		\begin{align*}\int \chi_{[x(1), \ldots, x(k)] \cap T^{-j}[0]}(z) \mathrm{d} \mu(z)	& = \begin{cases}
				2^{-k}	& \textrm{if } x(1 + j) = 0 , \\
				0	& \textrm{if } ,
			\end{cases} \\
			\int \chi_{[y(1), \ldots, y(k)] \cap T^{-j}[0]}(z)	& = \begin{cases}
				2^{-k}	& \textrm{if } z(1 + j) = y(1 + j) = 0 , \\
				0	& \textrm{otherwise} .
			\end{cases}
		\end{align*}
		But since $x(1 + j) = 0 \iff y(i + j) = 1$, it follows that
		$$\int \left( \chi_{[x(1), \ldots, x(k)] \cap T^{-j}[0]}(z) + \chi_{[y(1), \ldots, y(k)] \cap T^{-j}[0]}(z) \right)\mathrm{d} \mu(z) = 2^{-k}$$
		for all $j = 0, 1, \ldots, k - 1$. Therefore
		\begin{align*}
				& \alpha_{B\left(x, y; 2^{-k}, 2^{-k} \right)} \left( \operatorname{Avg}_{F_k} f \right) \\
			=	& 2^{k - 1} \int \frac{1}{k} \sum_{j = 0}^{k - 1} \left( \chi_{[x(1), \ldots, x(k)] \cap T^{-j}[0]}(z) + \chi_{[y(1), \ldots, y(k)] \cap T^{-j}[0]}(z) \right)\mathrm{d} \mu(z) \\
			=	& 2^{k - 1} \frac{1}{k} \sum_{j = 0}^{k - 1} 2^{-k} \\
			=	& \frac{1}{2} .
		\end{align*}
		So $\limsup_{k \to \infty} \alpha_{B\left(x, y; 2^{-k}, 2^{-k} \right)} \left( \operatorname{Avg}_{F_k} f \right) = 1/2 \neq 1$.
	\end{Ex}
	
	\begin{Ex}\label{Give and take example}
		Looking at Theorem \ref{Multi-local convergence}, we could also ask whether the result could be generalized to somehow accommodate the case where $\lim_{k \to \infty} \operatorname{Avg}_{F_k} f\left( x^{(h)} \right)$ exists for all $h = 1, \ldots, n$, but is allowed to vary with $h$. However, we can construct examples of points $x, y \in X$, sequences of radii $(r_k)_{k = 1}^\infty, (s_k)_{k = 1}^\infty \in (0, 1)^\mathbb{N}$ decaying $(X, \rho, H, L, \mathbf{F})$-fast, and a function $f \in C(X)$ where $\lim_{k \to \infty} \operatorname{Avg}_{F_k} f \left( x \right), \lim_{k \to \infty} \operatorname{Avg}_{F_k} f\left( y \right)$ both exist, but $\lim_{k \to \infty} \alpha_{B\left(x, y; 2^{-k}, 2^{-k} \right)} \left( \operatorname{Avg}_{F_k} f \right)$ does not. Let $X, \rho, T, \mu, \mathbf{F}$ be as in Example \ref{Can't take lim sups}, but choose $x, y$ to be
		\begin{align*}
			x(i)	& = \begin{cases}
				0	& \textrm{if $i$ is even}, \\
				1	& \textrm{if $i$ is odd},
			\end{cases}	&
			y(i)	& = \begin{cases}
				0	& \textrm{if $i$ is divisible by $3$}, \\
				1	& \textrm{otherwise.}
			\end{cases}
		\end{align*}
		Let $f = \chi_{[0]}$. Then $\lim_{k \to \infty} \operatorname{Avg}_{F_k} f(x) = 1/2, \lim_{k \to \infty} \operatorname{Avg}_{F_k} f(y) = 1/3$. Construct sequences of natural numbers $(p_k)_{k = 1}^\infty, (q_k)_{k = 1}^\infty$ strictly increasing such that
		\begin{align*}
			\frac{2^{-p_k}}{2^{-p_k} + 2^{-q_k}}	& \geq \frac{k}{k + 1}	& (\textrm{for $k$ odd}) , \\
			\frac{2^{-q_k}}{2^{-p_k} + 2^{-q_k}}	& \geq \frac{k}{k + 1}	& (\textrm{for $k$ even}) .
		\end{align*}
		Therefore
		$$\lim_{n \to \infty} \frac{2^{-p_{2n - 1}}}{2^{-p_{2n - 1}} + 2^{-q_{2n - 1}}} = \frac{2^{-q_{2n}}}{2^{-p_{2n}} + 2^{-q_{2n}}} = 1 .$$
		Set $r_k = 2^{-p_k}, s_k = 2^{-q_k}$. We can see that $(r_k, s_k)_{k = 1}^\infty$ decays $(X, \rho, H, L, \mathbf{F})$-fast. By Lemma \ref{Phasing for multi-balls}, we have
		$$
		\alpha_{B(x, y; r_k, s_k)} \left( \operatorname{Avg}_{F_k} f \right) = \frac{2^{-p_k}}{2^{-p_k} + 2^{-q_k}} \alpha_{B(x; r_k)} \left( \operatorname{Avg}_{F_k} f \right) + \frac{2^{-q_k}}{2^{-p_k} + 2^{-q_k}} \alpha_{B(y; s_k)} \left( \operatorname{Avg}_{F_k} f \right) .
		$$
		Sampling along even $k$, we see that
		\begin{align*}
			& \lim_{n \to \infty} \alpha_{B(x, y; r_{2n}, s_{2n})} \left( \operatorname{Avg}_{F_{2n}} f \right)	\\
			=	& \lim_{n \to \infty} \frac{2^{-p_{2n}}}{2^{-p_{2n}} + 2^{-q_{2n}}} \alpha_{B(x; r_{2n})} \left( \operatorname{Avg}_{F_{2n}} f \right) + \frac{2^{-q_{2n}}}{2^{-p_{2n}} + 2^{-q_{2n}}} \alpha_{B(y; s_{2n})} \left( \operatorname{Avg}_{F_{2n}} f \right) \\
			=	& 0 \left( \lim_{n \to \infty} \alpha_{B(x; r_{2n})} \left( \operatorname{Avg}_{F_{2n}} f \right) \right) + 1 \left( \lim_{n \to \infty} \alpha_{B(y; s_{2n})} \left( \operatorname{Avg}_{F_{2n}} f \right) \right) \\
			=	& \frac{1}{3} ,
		\end{align*}
		where the limits in the last step are taken using Lemma \ref{Singly local pointwise reduction}. On the other hand, sampling along odd $k$, we see that
		\begin{align*}
			& \lim_{n \to \infty} \alpha_{B(x, y; r_{2n - 1}, s_{2n - 1})} \left( \operatorname{Avg}_{F_{2n - 1}} f \right) \\
			=	& \lim_{n \to \infty} \frac{2^{-p_{2n - 1}}}{2^{-p_{2n - 1}} + 2^{-q_{2n - 1}}} \alpha_{B(x; r_{2n - 1})} \left( \operatorname{Avg}_{F_{2n - 1}} f \right) \\
			& + \lim_{n \to \infty} \frac{2^{-q_{2n - 1}}}{2^{-p_{2n - 1}} + 2^{-q_{2n - 1}}} \alpha_{B(y; s_{2n - 1})} \left( \operatorname{Avg}_{F_{2n - 1}} f \right) \\
			=	&  1 \left( \lim_{n \to \infty} \alpha_{B(x; r_{2n - 1})} \left( \operatorname{Avg}_{F_{2n - 1}} f \right) \right) + 0 \left( \lim_{n \to \infty} \alpha_{B(y; s_{2n - 1})} \left( \operatorname{Avg}_{F_{2n - 1}} f \right) \right) \\
			=	& \frac{1}{2},
		\end{align*}
		where we again appeal to Lemma \ref{Singly local pointwise reduction} to take the limits at the end. Thus the sequence
		$$\left( \alpha_{B(x, y; r_k, s_k)} \left( \operatorname{Avg}_{F_k} f \right) \right)_{k = 1}^\infty$$
		is divergent.
	\end{Ex}
	
	The argument employed in Example \ref{Give and take example}, where we control the ``weight" we give several points at different points in the temporo-spatial differentiation, will have applications in Sections \ref{Phasing problems} and \ref{Chasing measures}. However, the following result also demonstrates that absent such tricks, we have predictable convergence behaviors.
	
	\begin{Thm}\label{Blending different limits}
		Let $\left( \bar{r}_k \right)_{k = 1}^\infty$ be a sequence that decays $(X, \rho, H, L, \mathbf{F})$-fast, and let $f \in C(X)$. Suppose $\bar{x} = \left( x^{(1)}, \ldots, x^{(n)} \right)$ is an $n$-tuple in $X$ such that
		$$C_h = \lim_{k \to \infty} \operatorname{Avg}_{F_k} f\left( x^{(h)} \right)$$
		exists for all $h = 1, \ldots, n$. Let $\mu$ be a Borel probability measure on $X$ for which $x^{(1)}, \ldots, x^{(n)} \in \operatorname{supp}(\mu)$. Suppose further that
		$$
		D_h = \lim_{k \to \infty} \frac{\mu\left( B \left( x^{(h)} ; r_k^{(h)} \right) \right)}{\mu\left( B \left( x^{(1)} ; r_k^{(1)} \right) \right) + \cdots + \mu\left( B \left( x^{(n)} ; r_k^{(n)} \right) \right)}
		$$
		exists for all $h = 1, \ldots, n$.
		Then
		$$
		\lim_{k \to \infty} \alpha_{B (\bar{x} ; \bar{r}_k)} \left( \operatorname{Avg}_{F_k} f \right) =\sum_{h = 1}^n D_h C_h.
		$$
	\end{Thm}
	
	\begin{proof}
		This follows immediately from Lemmas \ref{Singly local pointwise reduction}, \ref{Phasing for multi-balls}.
	\end{proof}
	
	\section{Preliminaries from ergodic optimization}\label{EO}
	
	Here we prove a generalization of a result of O. Jenkinson \cite[Proposition 2.1]{Jenkinson} to the setting of actions of amenable topological groups. Our method of proof closely resembles Jenkinson's, but requires that we attend to a few extra details.
	
	Throughout this section, $T : G \curvearrowright X$ will be an action of a discrete amenable group $G$ on a compact metrizable space $X$ by homeomorphisms, and $f \in C_\mathbb{R}(X)$ will be a real-valued continuous function on $X$. Let $\mathbf{F} = (F_k)_{k = 1}^\infty$ be a \Folner \space sequence for $G$. Define the set $\operatorname{Reg}(f)$ by
	$$\operatorname{Reg}(f) = \left\{ x \in X : \textrm{$\lim_{k \to \infty} \operatorname{Avg}_{F_k} f(x)$ exists} \right\} .$$
	
	We define the following values:
	\begin{align*}
		\overline{a}(f)	& : = \sup \left\{ \int f \mathrm{d} \nu : \nu \in \mathcal{M}_T(X) \right\} ,	\\
		\underline{a}(f)	& : = \inf \left\{ \int f \mathrm{d} \nu : \nu \in \mathcal{M}_T(X) \right\} , \\
		\overline{b}(f)	& : = \sup \left\{ \lim_{k \to \infty} \operatorname{Avg}_{F_k} f(x) : x \in \operatorname{Reg}(f) \right\} ,\\
		\underline{b}(f)	& : = \inf \left\{ \lim_{k \to \infty} \operatorname{Avg}_{F_k} f(x) : x \in \operatorname{Reg}(f) \right\} , \\
		\overline{c}(f)	& : = \sup \left\{ \limsup_{k \to \infty} \operatorname{Avg}_{F_k} f(x) : x \in X \right\} , \\
		\underline{c}(f)	& : = \inf \left\{ \limsup_{k \to \infty} \operatorname{Avg}_{F_k} f(x) : x \in X \right\} , \\
		\overline{d}(f)	& : = \lim_{k \to \infty} \left( \sup \left\{ \operatorname{Avg}_{F_k} f(x) : x \in X \right\} \right) \\
		\underline{d}(f)	& : = \lim_{k \to \infty} \left( \inf \left\{ \operatorname{Avg}_{F_k} f(x) : x \in X \right\} \right) .
	\end{align*}
	We write $\overline{b}(f) = - \infty , \underline{b}(f) = + \infty$ if $\operatorname{Reg}(f) = \emptyset$. We will show in Theorem \ref{Jenkinson's ergodic extrema} that $\underline{d}(f) , \overline{d}(f)$ are well-defined.
	
	The following result is elementary, but will be relevant for much of this article, so we state and prove it here.
	
	\begin{Lem}\label{Invariance}
		Let $(F_k)_{k = 1}^\infty$ be a \Folner \space sequence for a group $G$, and let $(\beta_k)_{k = 1}^\infty$ be a sequence of Borel probability measures on $X$. Then if $k_1 < k_2 < \cdots$ is a sequence of natural numbers such that $\nu = \lim_{\ell \to \infty} \beta_{k_\ell} \circ \operatorname{Avg}_{F_{k_\ell}}$ exists, then $\nu \in \mathcal{M}_T(X)$. In particular, if $G$ is amenable, then $\mathcal{M}_T(X) \neq \emptyset$.
	\end{Lem}
	
	\begin{proof}
		Assume WLoG that $k_\ell = \ell$ for all $\ell \in \mathbb{N}$. Let $f \in C(X) , g \in G$.
		\begin{align*}
			\left| \int f \mathrm{d} \nu - \int T_g f \mathrm{d} \nu \right|	& = \lim_{k \to \infty} \left| \left( \frac{1}{\left|F_{k} \right|} \sum_{h \in F_{k}} \int T_h f \mathrm{d} \beta_k \right) - \left( \frac{1}{\left|F_{k} \right|} \sum_{h' \in g F_{k}} \int T_{h'} f \mathrm{d} \beta_k \right) \right| \\
			& = \lim_{k \to \infty} \frac{1}{|F_k|} \left| \left( \sum_{h \in F_k \setminus g F_k} T_h f \right) - \left( \sum_{h' \in g F_k \setminus F_k} T_{h'} f \right) \right| \\
			& \leq \limsup_{k \to \infty} \frac{|F_k \Delta g F_k|}{|F_k|} \|f\|_{C(X)} \\
			& = 0 .
		\end{align*}
		
		To prove that $\mathcal{M}_T(X) \neq \emptyset$, consider any Borel probability measure $\beta$ on $X$, and use the weak*-compactness of $\mathcal{M}(X)$ to extract a convergent subsequence from $\left( \beta \circ \operatorname{Avg}_{F_k} \right)_{k = 1}^\infty$. The limit of that convergent subsequence will be $T$-invariant.
	\end{proof}
	
	\begin{Def}
		Let $\nu \in \mathcal{M}_T(X)$, and $f \in C(X)$. A point $x \in X$ is called \emph{$(f, \mathbf{F}, \nu)$-typical} if $\lim_{k \to \infty} \operatorname{Avg}_{F_k} f(x) = \int f \mathrm{d} \nu$.
	\end{Def}
	
	\begin{Thm}\label{Jenkinson's ergodic extrema}
		Suppose $f \in C_\mathbb{R} (X)$. Then the values $\overline{a}(f), \underline{a}(f), \overline{c}(f), \underline{c}(f), \overline{d}(f), \underline{d}(f)$ are all well-defined real numbers, and
		\begin{align*}
			\overline{b}(f) \leq \overline{c}(f) = \overline{a}(f)	& = \overline{d}(f) , \\
			\underline{b}(f) \geq \underline{c}(f) = \underline{a}(f)	& = \underline{d}(f) .
		\end{align*}
		Furthermore, if for every ergodic measure $\theta \in \partial_e \mathcal{M}_T(X)$ exists an $(f, \mathbf{F}, \theta)$-typical point, then $$\overline{a}(f) = \overline{b}(f) = \overline{c}(f) = \overline{d}(f), \; \underline{a}(f) = \underline{b}(f) = \underline{c}(f) = \underline{d}(f) .$$
	\end{Thm}
	
	\begin{proof}
		We will only prove the inequalities and identities for $\overline{a}, \overline{b}, \overline{c}, \overline{d}$, since the analogous relations between $\underline{a}, \underline{b}, \underline{c}, \underline{d}$ can be proven in a parallel fashion.
		
		The well-definedness of $\overline{a}(f)$ follows from the weak*-compactness of $\mathcal{M}_T(X)$. We also know a priori that $\overline{c}(f) \leq \|f\|_{C(X)}$, and thus $\overline{c}(f)$ is well-defined.
		
		It still remains to prove that $\overline{d}(f)$ is well-defined, which we will accomplish by proving that $\overline{d}(f) = \overline{a}(f)$.
		
		For each $k \in \mathbb{N}$, choose $x_k \in X$ such that $\operatorname{Avg}_{F_k} f(x_k) = \sup \left\{ \operatorname{Avg}_{F_k} f(x) : x \in X \right\}$. Let $\mu_k$ be the Borel probability measure on $X$ defined by
		$$\int g \mathrm{d} \mu_k = \operatorname{Avg}_{F_k} f(x_k) .$$
		Let $\left( \mu_{k_\ell} \right)_{\ell = 1}^\infty$ be a weak*-convergent subsequence converging to the measure $\mu$. Then since $\mathbf{F}$ is \Folner, it follows from Lemma \ref{Invariance} that $\mu$ is $T$-invariant. Thus
		$$\overline{a}(f) \geq \int f \mathrm{d} \mu = \lim_{\ell \to \infty} \int \operatorname{Avg}_{F_{k_\ell}} f \mathrm{d} \mu_{k_\ell} = \lim_{\ell \to \infty} \left( \sup \left\{ \operatorname{Avg}_{F_{k_\ell}} f(x) : x \in X \right\} \right) . $$
		On the other hand, we know that if $\nu \in \mathcal{M}_T(X)$, then
		$$\int f \mathrm{d} \nu = \int \operatorname{Avg}_{F_{k_\ell}} f \mathrm{d} \nu \leq \sup \left\{ \operatorname{Avg}_{F_{k_\ell}} f(x) : x \in X \right\} , $$
		and thus taking $\ell \to \infty$ tells us that $\int f \mathrm{d} \nu \leq \int f \mathrm{d} \mu$. Therefore this measure $\mu$ is $f$-maximizing, meaning that $\overline{a}(f) = \int f \mathrm{d} \mu = \lim_{\ell \to \infty} \left( \sup \left\{ \operatorname{Avg}_{F_{k_\ell}} f(x) : x \in X \right\} \right)$. Since we know this holds true for \emph{any} weak*-convergent subsequence $\left(\mu_{k_\ell}\right)_{\ell = 1}^\infty$, and $\left( \mu_k \right)_{k = 1}^\infty$ takes values in the weak*-compact space $\mathcal{M}(X)$, we can conclude that $\overline{d}(f)$ is well-defined and equal to $\overline{a}(f)$.
		
		It follows immediately from the definitions that $\overline{b}(f) \leq \overline{c}(f)$, since
		\begin{align*}
			\overline{b}(f)	& = \sup \left\{ \lim_{k \to \infty} \operatorname{Avg}_{F_k} f(x) : x \in \operatorname{Reg}(f) \right\} \\
			& = \sup \left\{ \limsup_{k \to \infty} \operatorname{Avg}_{F_k} f(x) : x \in \operatorname{Reg}(f) \right\} \\
			& \leq \sup \left\{ \limsup_{k \to \infty} \operatorname{Avg}_{F_k} f(x) : x \in X \right\}	& = \overline{c}(f) .
		\end{align*}
		It similarly follows from definitions that $\overline{c}(f) \leq \overline{d}(f)$, since
		\begin{align*}
			\overline{c}(f)	& = \sup \left\{ \limsup_{k \to \infty} \operatorname{Avg}_{F_k} f(x) : x \in X \right\} \\
			& \leq \sup \left\{ \limsup_{k \to \infty} \operatorname{Avg}_{F_k} f(x_k) : (x_k)_{k = 1}^\infty \in X^\mathbb{N} \right\} \\
			& \leq \sup \left\{ \limsup_{k \to \infty} \left( \sup \left\{ \operatorname{Avg}_{F_k} f(x) : x \in X \right\} \right) : (x_k)_{k = 1}^\infty \in X^\mathbb{N} \right\} \\
			& = \limsup_{k \to \infty} \left( \sup \left\{ \operatorname{Avg}_{F_k} f(x) : x \in X \right\} \right)	& = \overline{d}(f) .
		\end{align*}
		
		Next we show that $\overline{a}(f) \leq \overline{c}(f)$. Let $k_1 < k_2 < \cdots$ such that $\left( F_{k_\ell} \right)_{\ell = 1}^\infty$ is a \emph{tempered} \Folner \space subsequence, a subsequence which exists by Lemma \ref{Tempered subsequence}. Let $\theta \in \partial_e \mathcal{M}_T(X)$. Then by the Lindenstrauss Ergodic Theorem, there exists $x \in X$ such that $\lim_{\ell \to \infty} \operatorname{Avg}_{F_{k_\ell}} f(x) = \int f \mathrm{d} \theta$. Therefore
		$$
		\int f \mathrm{d} \theta = \lim_{\ell \to \infty} \operatorname{Avg}_{F_{k_\ell}} f(x) \leq \limsup_{k \to \infty} \operatorname{Avg}_{F_k} f(x) \leq \overline{c}(f) .
		$$
		Suppose $\nu \in \mathcal{M}_T(X)$, and let $(\theta_x)_{x \in X}$ be the ergodic decomposition of $T : G \curvearrowright X$. Then
		$$
		\int f \mathrm{d} \nu = \int \left( \int f \mathrm{d} \theta_x \right) \mathrm{d} \nu(x) \leq \int \overline{c}(f) \mathrm{d} \nu(x) = \overline{c}(f) .
		$$
		Taking the supreumum over $\nu \in \mathcal{M}_T(X)$ confirms that $\overline{a}(f) \leq \overline{c}(f)$.
		
		Now assume that for every ergodic measure $\theta \in \partial_e \mathcal{M}_T(X)$ exists $x_\theta \in X$ such that $\int f \mathrm{d} \theta = \lim_{k \to \infty} \operatorname{Avg}_{F_k} f(x_\theta).$ We prove that $\overline{a}(f) \leq \overline{b}(f)$. To begin, we'll prove that $\int f \mathrm{d} \theta \leq \overline{b}(f)$ for all ergodic $\theta \in \partial_e \mathcal{M}_T(X)$, and then use the ergodic decomposition to extrapolate to the general case.
		
		First, consider the case where $\theta$ is an ergodic measure in $\mathcal{M}_T(X)$. Then there exists $x_\theta \in X$ such that
		$$\int f \mathrm{d} \theta = \lim_{k \to \infty} \operatorname{Avg}_{F_k} f\left(x_\theta\right) \leq \overline{b}(f) .$$
		Now suppose $\nu \in \mathcal{M}_T(X)$, and let $(\theta_x)_{x \in X}$ be the ergodic decomposition of $T : G \curvearrowright X$. Then
		$$\int f \mathrm{d} \nu = \int \left(\int f \mathrm{d} \theta_x\right) \mathrm{d} \nu(x) \leq \int \overline{b}(f) \mathrm{d} \nu(x) = \overline{b}(f) .$$
		Taking the supremum over $\nu \in \mathcal{M}_T(X)$ confirms that $\overline{a}(f) \leq \overline{b}(f)$.
	\end{proof}
	
	What remains unclear to us at this point is whether $\overline{a}(f) \leq \overline{b}(f), \; \underline{b}(f) \leq \underline{a}(f)$ in general. However, there are several general cases where we know the answer to be yes.
	\begin{itemize}
		\item If $\underline{a}(f) = \overline{a}(f)$, then every $x \in X$ is an $(f, \mathbf{F}, \nu)$-typical point for all $\nu \in \mathcal{M}_T(X)$. In particular, this will occur for all $f \in C_\mathbb{R}(X)$ if $T : G \curvearrowright X$ is uniquely ergodic.
		\item If $\mathbf{F}$ is tempered, then the Lindenstrauss Ergodic Theorem implies that the set of $(f, \mathbf{F}, \theta)$-typical points is of probability $1$ with respect to $\theta$ for ergodic $\theta$, and a fortiori, that the set is nonempty. This holds in particular if $G = \mathbb{Z}$ and $F_k = \{0, 1, \ldots, k - 1\}$ for all $k \in \mathbb{N}$, which is the setting of the classical Birkhoff Ergodic Theorem.
	\end{itemize}
	
	\begin{Cor}
		The values $\overline{c}(f), \underline{c}(f), \overline{d}(f), \underline{d}(f)$ are independent of the choice of \Folner \space sequence $\mathbf{F}$, and $\overline{b}(f) , \underline{b}(f)$ are independent of the choice of \emph{tempered} \Folner \space sequence.
	\end{Cor}
	
	\begin{proof}
		The first claim follows from the fact that $\overline{a}(f), \underline{a}(f)$ are independent of $\mathbf{F}$, combined with Theorem \ref{Jenkinson's ergodic extrema}. The second claim follows from the fact that if $\mathbf{F}$ is a \emph{tempered} \Folner \space sequence, then by the Lindenstrauss Ergodic Theorem, every ergodic measure $\theta \in \partial_e \mathcal{M}_T(X)$ admits an $(f, \mathbf{F}, \theta)$-typical point, meaning Theorem \ref{Jenkinson's ergodic extrema} tells us that $\overline{b}(f) = \overline{a}(f) , \underline{b}(f) = \underline{a}(f)$.
	\end{proof}
	
	\section{Pathological multi-local temporo-spatial differentiations of individual functions}\label{Phasing problems}
	
	This section is motivated by the following question: Given a real-valued function $f \in C_\mathbb{R}(X)$, what possible sets $\mathcal{K}$ can be realized as
	$$
	\mathcal{K} = \left\{ \lim_{\ell \to \infty} \alpha_{B\left(x, y; r_{k_\ell}, s_{k_\ell}\right)} \left( \operatorname{Avg}_{F_{k_\ell}} f \right) : k_1 < k_2 < \cdots, \; \textrm{$\lim_{\ell \to \infty} \alpha_{B\left(x, y; r_{k_\ell}, s_{k_\ell}\right)} \left( \operatorname{Avg}_{F_{k_\ell}} f \right)$ exists} \right\}
	$$
	through judicious choices of $(\bar{x}; \bar{r}_k)_{k = 1}^\infty$? If $\mathcal{K}$ is non-singleton, then the temporo-spatial differentiation will of course be divergent.
	
	Before constructing these pathological temporo-spatial differentiations, we define a measure-theoretic property which will be important to us in this section.
	
	\begin{Def}
		Let $(X, \rho)$ be a compact metric space, and let $\mu$ be a Borel probability measure on $X$. We say that $\mu$ \emph{neglects shells} if
		\begin{align*}
			\mu \left( \left\{ y \in X : \rho(x, y) = r \right\} \right)	& = 0	& (\forall x \in X, r \in [0, \infty)) .
		\end{align*}
	\end{Def}
	
	A probability measure which neglects shells is automatically non-atomic, but the converse is false. Consider the case of $X = \left\{ (a, b) \in \mathbb{R}^2 : a^2 + b^2 \leq 2 \right\}$ with the standard Euclidean metric. Let $\mu$ be the Borel probability measure
	$$\mu(E) = \frac{1}{ \mathcal{H}^1(S) } \mathcal{H}^1(S \cap E) ,$$
	where $\mathcal{H}^1$ is the $1$-dimensional Hausdorff measure and $S = \left\{ (a, b) \in \mathbb{R}^2 : a^2 + b^2 = 1 \right\}$ is the unit circle in $\mathbb{R}^2$. Then this $\mu$ is non-atomic, but does not neglect shells.
	
	\begin{Thm}\label{Continuity for measures that neglect shells}
		The following conditions are equivalent.
		\begin{enumerate}[label=(\roman*)]
			\item The function $\phi : X \times [0, \infty) \to [0, 1]$ defined by
			$$\phi(x, r) = \mu(B(x; r))$$
			is continuous.
			\item $\mu$ neglects shells.
		\end{enumerate}
	\end{Thm}
	
	\begin{proof}
		(i)$\Rightarrow$(ii): Suppose that $\phi$ is continuous, and fix $x \in X, r \in [0, \infty)$. Let $r_k = r + 1/k$ for all $k \in \mathbb{N}$. By downward continuity of measures, we know that
		$$\lim_{k \to \infty} \phi(x, r_k) = \mu \left( \{y \in X : \rho(x, y) \leq r \} \right) = \phi(x, r) + \mu(\left\{ y \in X : \rho(x, y) = r \right\}) .$$
		If $\lim_{k \to \infty} \phi(x, r_k) = \phi(x, r)$, then $\mu(\left\{ y \in X : \rho(x, y) = r \right\}) = 0$.
		
		(ii)$\Rightarrow$(i): Suppose that $\mu$ neglects shells, and let $(x_k, r_k)_{k = 1}^\infty$ be a sequence in $X \times [0, \infty)$ converging to $(x, r)$. Let $f_k, f \in L^\infty(X, \mu)$ be the functions
		\begin{align*}
			f_k	& = \chi_{B(x_k; r_k)} , \\
			f	& = \chi_{B(x; r)} .
		\end{align*}
		We claim that $f_k \to f$ pointwise on $\left\{ y \in X : \rho(x, y) \neq r \right\}$, which under the assumption that $\mu$ neglects shells constitutes convergence pointwise almost everywhere. If we can prove that, then we can appeal to the Dominated Convergence Theorem (using the constant function $1$ as a dominator) to conclude that $\phi(x_k, r_k) = \int f_k \mathrm{d} \mu \stackrel{k \to \infty}{\to} \int f \mathrm{d} \mu = \phi(x, r)$, i.e. that $\phi$ is (sequentially) continuous.
		
		First, consider the case where $\rho(x, y) < r$. Set $\epsilon = r - \rho(x, y)$. Then there exist $K_1, K_2 \in \mathbb{N}$ such that
		\begin{align*}
			k	& \geq K_1	& \Rightarrow |r_k - r|	& < \frac{\epsilon}{2} , \\
			k	& \geq K_2	& \Rightarrow \rho(x_k, x)	& < \frac{\epsilon}{2} .
		\end{align*}
		If $k \geq K_1$, then $r_k > r - \frac{\epsilon}{2}$. Set $K = \max \{K_1, K_2\}$, and suppose that $k \geq K$. Then
		\begin{align*}
			\rho(y, x_k)	& \leq \rho(y, x) + \rho(x, x_k) \\
			& < \rho(y, x) + \frac{\epsilon}{2} \\
			& = r - \frac{\epsilon}{2} \\
			& < r_k .
		\end{align*}
		Thus if $k \geq K$, then $f_k(y) = 1 = f(y)$. Therefore $\lim_{k \to \infty} f_k(y) = f(y)$ for $y \in B(x ; r)$.
		
		Second, consider the case where $\rho(x, z) > r$. Set $\delta = \min \left\{\rho(x, z) - r , \frac{\rho(x, z)}{2} \right\}$, and choose $L_1, L_2 \in \mathbb{N}$ such that
		\begin{align*}
			k	& \geq L_1	& \Rightarrow |r_k - r|	& < \frac{\delta}{2} , \\
			k	& \geq L_2	& \Rightarrow \rho(x_k, x)	& < \frac{\delta}{2} .
		\end{align*}
		Set $L = \max\{L_1, L_2\}$, and consider $k \geq L$. Then
		\begin{align*}
			\rho(z, x_k)	& \geq \left| \rho(z, x) - \rho(x, x_k) \right| \\
			& = \rho(z, x) - \rho(x, x_k) \\
			& > \rho(z, x) - \frac{\delta}{2} \\
			& > r + \delta - \frac{\delta}{2} \\
			& = r + \frac{\delta}{2} \\
			& > r_k .
		\end{align*}
		Thus if $k \geq L$, then $f_k(z) = 0 = f(z)$. Therefore $\lim_{k \to \infty} f_k(z) = f(z)$ for $\rho(z, x) > r$. This completes the proof.
	\end{proof}
	
	The property of neglecting shells is very important to us in this article because of Lemma \ref{Controlling relative measures}, which is a valuable tool for several constructions that will follow in this section and the next.
	
	\begin{Lem}\label{Controlling relative measures}
		Let $\mu$ be a Borel probability measure on $X$ that neglects shells, and let $x^{(1)}, \ldots, x^{(n)} \in \operatorname{supp}(\mu)$. Let $\delta^{(1)}, \ldots, \delta^{(n)} > 0$, and fix $\lambda^{(1)}, \ldots, \lambda^{(n)} \in (0, 1)$ such that $\lambda^{(1)} + \cdots + \lambda^{(n)} = 1$. Then there exist $r^{(1)}, \ldots, r^{(n)} > 0$ such that $0 < r^{(h)} < \delta^{(h)},$ and
		\begin{align*}
			\frac{\mu \left( B \left( x^{(h)} ; r^{(h)} \right) \right)}{ \mu \left( B \left( x^{(1)} ; r^{(1)} \right) \right) + \cdots + \mu \left( B \left( x^{(n)} ; r^{(n)} \right) \right)}	& = \lambda^{(h)}	& (h = 1, \ldots, n) .
		\end{align*}
	\end{Lem}
	
	\begin{proof}
		Assume without loss of generality that
		$$
		\delta^{(h)} < \min_{1 \leq i < j \leq n} \rho \left( x^{(i)} , x^{(j)} \right) ,
		$$
		otherwise we can replace each $\delta^{(h)}$ with $\min \left\{ \delta^{(h)} , \frac{1}{4} \min_{1 \leq i < j \leq n} \rho \left( x^{(i)} , x^{(j)} \right) \right\}$.
		
		Choose real numbers $a^{(1)}, \ldots, a^{(h)} \in (0, 1)$ such that
		\begin{align*}
			\frac{a^{(h)}}{a^{(1)} + \cdots + a^{(n)}}	& = \lambda^{(h)} , \\
			a^{(1)}	& < \mu\left(B\left(x^{(h)}; \delta^{(h)}\right)\right) \\
		\end{align*}
		for all $h = 1, \ldots, n$. The tuple $\left(a^{(1)},\ldots, a^{(n)}\right) \in (0, 1)^n$ can be found along the line segment $\left\{ \left( t \lambda^{(1)}, \ldots, t \lambda^{(n)} \right) : t \in (0, 1) \right\}$. We know that $\mu\left(B\left(x^{(h)}; \delta^{(h)}\right)\right) > 0$ because we assumed that $x^{(h)} \in \operatorname{supp}(\mu)$. Then by Theorem \ref{Continuity for measures that neglect shells} and the Intermediate Value Theorem, there exist $r^{(h)} \in \left( 0 , \delta^{(h)} \right)$ such that
		\begin{align*}
			\mu \left( B \left( x^{(h)} ; r^{(h)} \right) \right)	& = a^{(h)}	& (h = 1, \ldots, n) ,
		\end{align*}
		and therefore
		\begin{align*}
			\frac{\mu \left( B \left( x^{(h)} ; r^{(h)} \right) \right)}{ \mu \left( B \left( x^{(1)} ; r^{(1)} \right) \right) + \cdots + \mu \left( B \left( x^{(n)} ; r^{(n)} \right) \right)}	& = \lambda^{(h)}	& (h = 1, \ldots, n) .
		\end{align*}
	\end{proof}
	
	\begin{Thm}\label{Sandwiching compact sets}
		Let $x, y \in X$ such that
		\begin{align*}
			u	& = \lim_{k \to \infty} \operatorname{Avg}_{F_k} f(x) , \\
			v	& = \lim_{k \to \infty} \operatorname{Avg}_{F_k} f(y)
		\end{align*}
		exist, where $u \leq v$. Suppose $\mathcal{K} \subseteq [u, v]$ is a nonempty compact subset. Let $\mu$ be a fully supported Borel probability measure on $X$ that neglects shells. Then there exist sequences $(r_k)_{k = 1}^\infty , (s_k)_{k = 1}^\infty$ of positive numbers such that
		$$
		\mathcal{K} = \operatorname{LS} \left( \left( \alpha_{B\left(x, y; r_{k}, s_{k}\right)} \left( \operatorname{Avg}_{F_{k}} f \right) \right)_{k = 1}^\infty \right)
		$$
	\end{Thm}
	
	\begin{proof}
		Let $P = \{p_i : i \in I\} \subseteq \mathcal{K}$ be a countable dense subset of $\mathcal{K}$ enumerated by the countable indexing set $I$, and let $\mathscr{N} = \left\{ \mathcal{N}_i : i \in I \right\}$ be a partition of $\mathbb{N}$ into countably many infinite subsets, also enumerated by $I$. For convenience, write $i(k)$ for the $i \in I$ such that $k \in \mathcal{N}_i$.
		
		For each $i \in I$, choose $\lambda_i \in [0, 1]$ such that
		$$p_i = \lambda_i u + (1- \lambda_i) v .$$
		For each $k \in \mathbb{N}$, choose $t_k \in (0, 1)$ such that
		$$
		|t_k - \lambda_{i(k)}|	\leq 1/k .
		$$
		
		Using the uniform continuity of $\operatorname{Avg}_{F_k} f$ and Lemma \ref{Controlling relative measures}, choose $(r_k, s_k)_{k = 1}^\infty$ such that
		\begin{align*}
			\rho\left(w, z\right) \leq \max \{ r_k , s_k \}	& \Rightarrow \left| \operatorname{Avg}_{F_k} f\left( w \right) - \operatorname{Avg}_{F_k} f \left( z \right) \right| \leq 1/k	& \left( \forall w, z \in X \right) , \\
			\frac{\mu(B(x; r_k))}{\mu(B(x; r_k)) + \mu(B(y; s_k))}	& = t_k
		\end{align*}
		for all $k \in \mathbb{N}$.
		
		For $k \in \mathbb{N}$, we have that
		\begin{align*}
			& \left| p_{i(k)} - \alpha_{B(x, y; r_k, s_k)}\left( \operatorname{Avg}_{F_k} f \right) \right| \\
			=	& \left| p_{i(k)} - \left( t_k \alpha_{B(x; r_k)}\left( \operatorname{Avg}_{F_k} f \right) + (1 - t_k) \alpha_{B(y; s_k)}\left( \operatorname{Avg}_{F_k} f \right)\right)\right| \\
			\leq	& \left| p_{i(k)} - \left( t_k \left( \operatorname{Avg}_{F_k} f(x) \right) + (1 - t_k) \left( \operatorname{Avg}_{F_k} f(y) \right)\right)\right| \\
			& + t_k \left| \alpha_{B(x; r_k)}\operatorname{Avg}_{F_k} (f(x) - f) \right| + (1 - t_k) \left| \alpha_{B(y; s_k)}\operatorname{Avg}_{F_k} (f(y) - f) \right| \\
			\leq	& \left| p_{i(k)} - \left( t_k \left( \operatorname{Avg}_{F_k} f(x) \right) + (1 - t_k) \left( \operatorname{Avg}_{F_k} f(y) \right)\right)\right| + \frac{1}{k} \\
			=	& \left| p_{i(k)} - \left( t_k u + (1 - t_k) v\right)\right| + t_k \left| u - \operatorname{Avg}_{F_k} f(x) \right| + (1 - t_k) \left| v - \operatorname{Avg}_{F_k} f(y) \right| + \frac{1}{k} \\
			=	& \left| \lambda_{i(k)} u + (1 - \lambda_{i(k)}) v - \left( t_k u + (1 - t_k) v\right)\right| \\
			& + t_k \left| u - \operatorname{Avg}_{F_k} f(x) \right| + (1 - t_k) \left| v - \operatorname{Avg}_{F_k} f(y) \right| + \frac{1}{k} \\
			\leq	& \left| \left(\lambda_{i(k)} - t_k \right) u \right| + \left| \left((1 -\lambda_{i(k)}) - (1 - t_k) \right) v \right| \\
			& + t_k \left| u - \operatorname{Avg}_{F_k} f(x) \right| + (1 - t_k) \left| v - \operatorname{Avg}_{F_k} f(y) \right| + \frac{1}{k} \\
			=	&\left| \left(\lambda_{i(k)} - t_k \right) u \right| + \left| \left( \lambda_{i(k)} - t_{k} \right) v \right| \\
			& + t_{k} \left| u - \operatorname{Avg}_{F_k} f(x) \right| + (1 - t_{k}) \left| v - \operatorname{Avg}_{F_k} f(y) \right| + \frac{1}{k} \\
			\leq	& \frac{1}{k} |u| + \frac{1}{k} |v| + t_{k} \left| u - \operatorname{Avg}_{F_k} f(x) \right| + (1 - t_{k}) \left| v - \operatorname{Avg}_{F_k} f(y) \right| + \frac{1}{k} \\
			\leq	& \frac{|u|}{k} + \frac{|v|}{k} + \left| u - \operatorname{Avg}_{F_k} f(x) \right| + \left| v - \operatorname{Avg}_{F_k} f(y) \right| + \frac{1}{k}
		\end{align*}
		
		We now claim that
		\begin{align*}
			& \mathcal{K} \\
			& = \left\{ \lim_{\ell \to \infty} \alpha_{B\left(x, y; r_{k_\ell}, s_{k_\ell}\right)} \left( \operatorname{Avg}_{F_{k_\ell}} f \right) : k_1 < k_2 < \cdots, \; \textrm{$\lim_{\ell \to \infty} \alpha_{B\left(x, y; r_{k_\ell}, s_{k_\ell}\right)} \left( \operatorname{Avg}_{F_{k_\ell}} f \right)$ exists} \right\} .
		\end{align*}
		
		We will prove the two sets contain each other, and thus are equal. First, let $q \in \mathcal{K}$, and choose a sequence $\left( p_{i_\ell} \right)_{\ell = 1}^\infty$ in $S$ such that $\left| q - p_{i_\ell} \right| < 1 / \ell$ for all $\ell \in \mathbb{N}$. For each $\ell \in \mathbb{N}$, recursively choose $k_\ell > \max \{k_1, \ldots, k_{\ell - 1}\}$ such that
		\begin{align*}
			\left| u - \operatorname{Avg}_{F_{k_\ell}} f(x) \right|	& < 1/\ell , \\
			\left| v - \operatorname{Avg}_{F_{k_\ell}} f(y) \right|	& < 1/\ell , \\
			k_\ell	& \in \mathcal{N}_{i_\ell} .
		\end{align*}
		Then
		\begin{align*}
			\left| q - \alpha_{B\left(x, y; r_{k_\ell}, s_{k_\ell} \right)}\left( \operatorname{Avg}_{F_{k_\ell}} f \right) \right|	& \leq \left| q - p_{i_\ell} \right| + \left| p_{i_\ell} - \alpha_{B\left(x, y; r_{k_\ell}, s_{k_\ell}\right)}\left( \operatorname{Avg}_{F_{k_\ell}} f \right) \right| \\
			& \leq \frac{1}{\ell} + \frac{1}{k_\ell} |u| + \frac{1}{k_\ell} |v| + \left| u - \operatorname{Avg}_{F_{k_\ell}} f(x) \right| \\
				& + \left| v - \operatorname{Avg}_{F_{k_\ell}} f(y) \right| + \frac{1}{k_\ell} \\
			& \leq \frac{1}{\ell} + \frac{|u|}{\ell} + \frac{|v|}{\ell} + \frac{1}{\ell} + \frac{1}{\ell} + \frac{1}{\ell} \\
			& = \frac{4 + |u| + |v|}{\ell} \\
			& \stackrel{\ell \to \infty}{\to} 0 .
		\end{align*}
		Therefore
		$$q \in \operatorname{LS} \left( \left( \alpha_{B\left(x, y; r_{k}, s_{k}\right)} \left( \operatorname{Avg}_{F_{k}} f \right) \right)_{k = 1}^\infty \right) .$$
		
		Conversely, let $k_1 < k_2 < \cdots$ be an increasing sequence of natural numbers such that $q = \lim_{\ell \to \infty} \alpha_{B\left(x_{k_\ell}, y_{k_\ell}; r_{k_\ell}, y_{k_\ell}\right)} \left( \operatorname{Avg}_{F_{k_\ell}} f \right)$ exists. Fix $\epsilon > 0$, and choose $K \in \mathbb{N}$ sufficiently large that
		\begin{align*}
			\Rightarrow \left| u - \operatorname{Avg}_{F_k} f(x) \right|	& < \epsilon	&(\forall k \geq K) , \\
			\Rightarrow \left| v - \operatorname{Avg}_{F_k} f(y) \right|	& < \epsilon	& (\forall k \geq K) , \\
			\frac{\max \{|u|, |v|, 1\}}{K}	& < \epsilon , \\
			\left| q - \alpha_{B\left(x, y; r_{k_\ell}, s_{k_\ell}\right)}\left( \operatorname{Avg}_{F_{k_\ell}} f \right) \right|	& < \epsilon	& (\forall \ell \geq K) .
		\end{align*}
		Then if $\ell \geq K$, we have
		\begin{align*}
			\left| p_{i \left( k_\ell \right)} - q \right|	& \leq \left| p_{i\left( k_\ell \right)} - \alpha_{B\left(x, y; r_{k_\ell}, s_{k_\ell}\right)}\left( \operatorname{Avg}_{F_{k_\ell}} f \right) \right| + \left| \alpha_{B\left(x, y; r_{k_\ell}, s_{k_\ell}\right)}\left( \operatorname{Avg}_{F_{k_\ell}} f \right) - q \right| \\
			& \leq \frac{1}{k} |u| + \frac{1}{k} |v| + \left| u - \operatorname{Avg}_{F_k} f(x) \right| + \left| v - \operatorname{Avg}_{F_k} f(y) \right| + \frac{1}{k} + \epsilon \\
			& < 6 \epsilon .
		\end{align*}
		Therefore $\inf_{p \in \mathcal{K}} |p - q| < 6 \epsilon$. Since our choice of $\epsilon > 0$ was arbitrary, it follows that \linebreak$\inf_{p \in \mathcal{K}} |p - q| = 0$, and since $\mathcal{K}$ is compact, this implies that $q \in \mathcal{K}$.
	\end{proof}
	
	\begin{Cor}\label{Targeting compact sets between extremes}
		Suppose $G$ is an amenable group, and $\mathbf{F} = (F_k)_{k = 1}^\infty$ is a right \Folner \space sequence for $G$. Let $f \in C_\mathbb{R}(X)$ such that for every ergodic $\theta \in \partial_e \mathcal{M}_T(X)$ exists an $(f, \mathbf{F}, \theta)$-typical point. Let $\mathcal{K}$ be a compact subset of the compact interval
		$$\left[ \underline{a}(f) , \; \overline{a}(f) \right] .$$
		Let $\mu$ be a fully supported Borel probability measure on $X$ that neglects shells. Then there exist points $x, y \in X$ and sequences $(r_k)_{k = 1}^\infty , (s_k)_{k = 1}^\infty$ of positive numbers such that
		$$
		\mathcal{K} = \operatorname{LS} \left( \left( \alpha_{B\left(x, y; r_{k}, s_{k}\right)} \left( \operatorname{Avg}_{F_{k}} f \right) \right)_{k = 1}^\infty \right) .
		$$
	\end{Cor}
	
	\begin{proof}
		By \cite[Proposition 2.4-(iii)]{Jenkinson}, there exist ergodic Borel probability measures $\theta_1, \theta_2$ such that
		\begin{align*}
			\int f \mathrm{d} \theta_1	& = \underline{a}(f) , \\
			\int f \mathrm{d} \theta_2	& = \overline{a}(f) .
		\end{align*}	
		By hypothesis, there exist $x, y \in X$ such that
		\begin{align*}
			\lim_{\ell \to \infty} \operatorname{Avg}_{F_{k_\ell}} f(x)	& = \int f \mathrm{d} \theta_1 , \\
			\lim_{\ell \to \infty} \operatorname{Avg}_{F_{k_\ell}} f(y)	& = \int f \mathrm{d} \theta_2 .
		\end{align*}
		Apply Theorem \ref{Sandwiching compact sets}.
	\end{proof}
	
	\begin{Cor}
		Suppose $G$ is an amenable group, and $\mathbf{F} = (F_k)_{k = 1}^\infty$ is a right \Folner \space sequence for $G$. Let $f \in C_\mathbb{R}(X)$, and let $\mathcal{K}$ be a compact subset of the compact interval
		$$\left[ \underline{a}(f) , \; \overline{a}(f) \right] .$$
		Let $\mu$ be a fully supported Borel probability measure on $X$ that neglects shells. Then there exist points $x, y \in X$ and sequences $(r_k)_{k = 1}^\infty , (s_k)_{k = 1}^\infty$ of positive numbers such that
		$$
		\mathcal{K} \subseteq \operatorname{LS} \left( \left( \alpha_{B\left(x, y; r_{k}, s_{k}\right)} \left( \operatorname{Avg}_{F_{k}} f \right) \right)_{k = 1}^\infty \right) .
		$$
	\end{Cor}
	
	\begin{proof}
		Choose a tempered \Folner \space subsequence $\left( F_{k_\ell} \right)_{\ell = 1}^\infty$ of $\mathbf{F}$. By \cite[Proposition 2.4-(iii)]{Jenkinson}, there exist ergodic Borel probability measures $\theta_1, \theta_2$ such that
		\begin{align*}
			\int f \mathrm{d} \theta_1	& = \underline{a}(f) , \\
			\int f \mathrm{d} \theta_2	& = \overline{a}(f) .
		\end{align*}
		By the Lindestrauss Ergodic Theorem, there exist $x, y \in X$ such that
		\begin{align*}
			\lim_{\ell \to \infty} \operatorname{Avg}_{F_{k_\ell}} f(x)	& = \int f \mathrm{d} \theta_1 , \\
			\lim_{\ell \to \infty} \operatorname{Avg}_{F_{k_\ell}} f(y)	& = \int f \mathrm{d} \theta_2 .
		\end{align*}
		By Theorem \ref{Sandwiching compact sets}, there exist $\left(r_{k}\right)_{k = 1}^\infty, \left( s_k \right)_{k = 1}^\infty \in (0, \infty)^\mathbb{N}$ such that
		$$\mathcal{K} = \operatorname{LS} \left( \left( \alpha_{B\left(x, y; r_{k_\ell}, s_{k_\ell}\right)} \left( \operatorname{Avg}_{F_{k_\ell}} f \right) \right)_{\ell = 1}^\infty \right) .$$
		Then
		$$
		\mathcal{K} \subseteq \operatorname{LS} \left( \left( \alpha_{B\left(x, y; r_{k}, s_{k}\right)} \left( \operatorname{Avg}_{F_{k}} f \right) \right)_{k = 1}^\infty \right) .
		$$
	\end{proof}

	\begin{Thm}
		Suppose $G$ is an amenable group, and $\mathbf{F} = (F_k)_{k = 1}^\infty$ is a tempered \Folner \space sequence for $G$. Let $f \in C_\mathbb{R}(X)$, and let $\mathcal{K}$ be a compact subset of the compact interval
		$$\left[ \underline{a}(f) , \; \overline{a}(f) \right] .$$
		Let $\mu$ be a fully supported Borel probability measure on $X$ that neglects shells. Then there exist points $x, y \in X$ and sequences $(r_k)_{k = 1}^\infty , (s_k)_{k = 1}^\infty$ of positive numbers such that
		$$
		\mathcal{K} = \operatorname{LS} \left( \left( \alpha_{B\left(x, y; r_{k}, s_{k}\right)} \left( \operatorname{Avg}_{F_{k}} f \right) \right)_{k = 1}^\infty \right) .
		$$
	\end{Thm}
	
	\begin{proof}
		The Lindenstrauss Ergodic Theorem implies that for every ergodic $\theta \in \partial_e \mathcal{M}_T(X)$ exists an $(f, \mathbf{F}, \theta)$-typical point. Apply Corollary \ref{Targeting compact sets between extremes}.
	\end{proof}
	
	\section{Pathological multi-local temporo-spatial differentiations on $C(X)$}\label{Chasing measures}
	
	In this section, we consider a temporo-spatial differentiation $\left( \alpha_{C_k} \circ \operatorname{Avg}_{F_k} \right)_{k = 1}^\infty$ as a sequence in $\mathcal{M}_T(X)$. If $G$ is a discrete amenable group, and $\mathbf{F} = (F_k)_{k = 1}^\infty$ is a \Folner \space sequence, then Lemma \ref{Invariance} tells us that
	$$
	\operatorname{LS} \left( \left( \alpha_{C_k} \circ \operatorname{Avg}_{F_k} \right)_{k = 1}^\infty \right) \subseteq \mathcal{M}_T(X) 
	$$
	for all sequences $(C_k)_{k = 1}^\infty$ of measurable subsets of $X$ with positive measure.
	
	We are motivated here by the following question: Consider an action $T : G \curvearrowright X$ of a discrete amenable group $G$ on a compact metrizable space by $X$, where $X$ is endowed with a Borel probability measure $\mu$. Given a \Folner \space sequence $\mathbf{F} = (F_k)_{k = 1}^\infty$ for $G$, can we choose a sequence $(C_k)_{k = 1}^\infty$ of measurable subsets of $X$ with $\mu(C_k) > 0$ such that
	$$
	\operatorname{LS} \left( \left( \alpha_{C_k} \circ \operatorname{Avg}_{F_k} \right)_{k = 1}^\infty \right) = \mathcal{C} ,
	$$
	where $\mathcal{C}$ is some prescribed compact subset of $\mathcal{M}_T(X)$? If so, then can the $(C_k)_{k = 1}^\infty$ be chosen to fit some prescribed constraints?
	
	In this section, we provide positive answers for certain classes of $\mathcal{C}$. Throughout this section, assume that $G$ is a discrete amenable group and $\mathbf{F}$ is a \Folner \space sequence for $G$. We also assume that $T : G \curvearrowright X$ is a H\"older action with MoH\"oC $(H, L)$.
	
	\begin{Lem}\label{1-Lipschitz functions dense span}
		Let $\mathcal{L}  \subseteq C(X)$ denote the family of functions $f \in C(X)$ for which
		\begin{align*}
			|f(x) - f(y)|	& \leq \rho(x, y)	& (\forall x, y \in X) .
		\end{align*}
		Then $\mathcal{L}$ has dense span in $C(X)$.
	\end{Lem}
	
	\begin{proof}
		For $x_0 \in X$, set $\phi_{x_0} (x) = \rho(x, x_0)$. If $x, y \in X$, then by the Reverse Triangle Inequality we know
		$$
		|\phi_{x_0}(x) - \phi_{x_0}(y)| = |\rho(x, x_0) - \rho(y, x_0)| \leq \rho(x, y) .
		$$
		Thus the functions $\phi_{x_0}$ satisfy the prescribed Lipschitz condition, as does the constant function $1$. Furthermore, we know that $\left\{ \phi_{x_0} : x_0 \in X \right\}$ separates points, since if $x, y \in X, x \neq y$, then $0 = \phi_x(x) \neq \phi_x(y)$. Therefore by the Stone-Weierstrass Theorem, we know that $C(X)$ is densely spanned by finite products of elements in $\left\{ \phi_{x_0} : x \in X \right\} \cup \{1\} \subseteq \mathcal{L}$. We claim, however, that a product of elements in $\mathcal{L}$ is a scalar multiple of an element in $\mathcal{L}$. Let $f_1, f_2 \in \mathcal{L}$. Then
		\begin{align*}
			|f_1(x) f_2(x) - f_1(y) f_2(y)|	& = |f_1(x) f_2(x) - f_1(x) f_2(y) + f_1(x) f_2(y) - f_1(y) f_2(y)| \\
			& \leq |f_1(x)| \cdot |f_2(x) - f_2(y)| + |f_1(x) - f_1(y)| \cdot |f_2(y)| \\
			& \leq \|f_1\|_{C(X)} \cdot |f_2(x) - f_2(y)| + |f_1(x) - f_1(y)| \cdot \|f_2\|_{C(X)} \\
			& \leq \left( \|f_1\|_{C(X)} + \|f_2\|_{C(X)} \right) \rho(x, y) .
		\end{align*}
		Let $h = \frac{f_1 f_2} { \|f_1\|_{C(X)} + \|f_2\|_{C(X)} + 1 }$. Then $h \in \mathcal{L}$, so $f_1 f_2 = \left( \|f_1\|_{C(X)} + \|f_2\|_{C(X) } + 1 \right) h \in \mathbb{C} \mathcal{L}$. By an inductive argument, we can show that any finite product of elements of $\mathcal{L}$ is an element of $\mathbb{C} \mathcal{L}$. Therefore, the Stone-Weierstrass Theorem tells us that $C(X)$ is densely spanned by $\mathcal{L}$.
	\end{proof}
	
	\begin{Thm}\label{Targeting finite-dimensional subsets}
		Let $\theta^{(1)}, \ldots, \theta^{(n)} \in \partial_e \mathcal{M}_T(X)$ be a finite collection of ergodic measures on $X$, and let $\mathcal{C}$ be a compact subset of the convex hull of $\left\{ \theta^{(1)}, \ldots, \theta^{(n)} \right\}$. Suppose $\mathbf{F}$ is a tempered \Folner \space sequence, and that $\mu$ is a Borel probability measure on $X$ that neglects shells. Then there exist points $x^{(1)}, \ldots, x^{(n)}$ and sequences of radii $\left( r_k^{(1)} \right)_{k = 1}^\infty , \ldots, \left( r_k^{(n)} \right)_{k = 1}^\infty$ such that
		$$
		\operatorname{LS} \left( \left( \alpha_{B \left( x^{(1)} , \ldots, x^{(n)} ; r_k^{(1)} , \ldots, r_k^{(n)} \right)} \circ \operatorname{Avg}_{F_k} \right)_{k = 1}^\infty \right) = \mathcal{C} .
		$$
		Moreover, the set of $n$-tuples $\left( x^{(1)}, \ldots, x^{(n)} \right) \in X^n$ which admit such sequences $\left( r_k^{(1)}, \ldots, r_k^{(n)} \right)_{k = 1}^\infty$ is of full probability with respect to the product measure $\theta^{(1)} \times \cdots \times \theta^{(n)}$.
	\end{Thm}
	
	\begin{proof}
		Assume without loss of generality that $\theta^{(1)} , \ldots, \theta^{(n)}$ are distinct. By the Lindenstrauss Ergodic Theorem, there exist points $x^{(1)}, \ldots, x^{(n)} \in \operatorname{supp}(\mu)$ such that
		\begin{align*}
			\lim_{k \to \infty} \operatorname{Avg}_{F_k} f \left( x^{(h)} \right)	& = \int f \mathrm{d} \theta^{(h)}	& (h = 1, \ldots, n) .
		\end{align*}
		In fact, the Lindenstrauss Ergodic Theorem tells us that the set of such $\left( x^{(1)}, \ldots, x^{(n)} \right) \in X^n$ is of full measure with respect to $\theta^{(1)} \times \cdots \times \theta^{(n)}$. For the remainder of this proof, let $\bar{x} = \left( x^{(1)} , \ldots, x^{(n)} \right) \in X^n$ be such an $n$-tuple.
		
		For each $i \in I$, let $\bar{\lambda}_i = \left( \lambda_i^{(1)} , \ldots, \lambda_i^{(n)} \right) \in [0, 1]^n$ be such that $\nu_i = \sum_{h = 1}^n \lambda_i^{(h)} \theta^{(h)}$.
		
		Let $\mathscr{N} = \left\{ \mathcal{N}_i : i \in I \right\}$ be a partition of $\mathbb{N}$ into infinite subsets. For each $k \in \mathbb{N}$, set $i(k) \in I$ such that $k \in \mathcal{N}_{i(k)}$. For each $k \in \mathbb{N}$, choose $\bar{t}_k = \left( t_k^{(1)} , \ldots, t_k^{(n)} \right) \in (0, 1)^n$ such that
		\begin{align*}
			\sum_{h = 1}^n \left| t_k^{(h)} - \lambda_{i(k)}^{(h)} \right|	& < 1 / k , \\
			\sum_{h = 1}^n t_k^{(h)}	& = 1 .
		\end{align*}
		For each $k \in \mathbb{N}$, choose $\delta_k > 0$ such that
		$$
		\max_{g \in F_k} \left( L(g) \cdot \delta_k^{H(g)} \right) < 1 / k .
		$$
		Now for each $k \in \mathbb{N}$, use Lemma \ref{Controlling relative measures} to choose $\bar{r}_k = \left( r_k^{(1)} , \ldots, r_k^{(n)} \right) \in (0, 1)^n$ such that
		\begin{align*}
			\frac{\mu \left( B \left( x^{(h)} ; r_k^{(h)} \right) \right)}{ \mu \left( B \left( x^{(1)} ; r_k^{(1)} \right) \right) + \cdots + \mu \left( B \left( x^{(n)} ; r_k^{(n)} \right) \right)}	& = t_k^{(h)} , \\
			r_k^{(h)}	& < \delta_k , \\
			r_k^{(h)}	& < \frac{1}{3} \min \left\{ \rho \left( x^{(h_1)} , x^{(h_2)} \right) : 1 \leq h_1 < h_2 \leq n \right\} .
		\end{align*}
		The last condition ensures that the balls $\left\{ B \left( x^{(h)} ; r_k^{(h)} \right) : h = 1, \ldots, n \right\}$ are pairwise disjoint. Since the points $x^{(h)}$ each satisfy
		$$
		\lim_{k \to \infty} \operatorname{Avg}_{F_k} f \left( x^{(h)} \right) = \int f \mathrm{d} \theta^{(h)} ,
		$$
		for all $f \in C(X)$, and the measures $\theta^{(1)} , \ldots, \theta^{(h)}$ are distinct, it follows that the $x^{(1)} , \ldots, x^{(n)}$ are also distinct, meaning that $\min \left\{ \rho \left( x^{(h_1)} , x^{(h_2)} \right) : 1 \leq h_1 < h_2 \leq n \right\} > 0$.
		
		Let $\mathcal{L} \subseteq C(X)$ denote the family of all continuous functions $f$ on $X$ such that
		\begin{align*}
			|f(x) - f(y)|	& \leq \rho(x, y)	& (\forall x, y \in X),
		\end{align*}
		i.e. the $1$-Lipschitz functions $X \to \mathbb{C}$, and let $f \in \mathcal{L}$. Then
		\begin{align*}
			& \left| \alpha_{B \left( \bar{x}, \bar{r}_k \right)} \left( \operatorname{Avg}_{F_k} f \right) - \int f \mathrm{d} \nu_{i(k)} \right| \\
			[\textrm{Lem. \ref{Phasing for multi-balls}}]=	& \left| \left[ \sum_{h = 1}^n \frac{\mu\left( B \left( x^{(h)} ; r_k^{(h)} \right) \right)}{\sum_{u = 1}^n \mu \left( B \left( x^{(u)} ; r^{(u)} \right) \right)} \alpha_{B \left( x^{(h)} ; r_k^{(h)} \right)} \left( \operatorname{Avg}_{F_k} f \right) \right] - \int f \mathrm{d} \nu_{i(k)} \right| \\
			=	& \left| \left[ \sum_{h = 1}^n t_k^{(h)} \alpha_{B \left( x^{(h)} ; r_k^{(h)} \right)} \left( \operatorname{Avg}_{F_k} f \right) \right] - \int f \mathrm{d} \nu_{i(k)} \right| \\
			=	& \left| \sum_{h = 1}^n \left( t_k^{(h)} \alpha_{B \left( x^{(h)} ; r_k^{(h)} \right)} \left( \operatorname{Avg}_{F_k} f\right) - \lambda_{i(k)}^{(h)} \int f \mathrm{d} \theta^{(h)} \right) \right| \\
			\leq	& \sum_{h = 1}^n \left| t_k^{(h)} \alpha_{B \left( x^{(h)} ; r_k^{(h)} \right)} \left( \operatorname{Avg}_{F_k} f\right) - \lambda_{i(k)}^{(h)} \int f \mathrm{d} \theta^{(h)} \right| \\
			\leq	& \sum_{h = 1}^n \left[ \left| t_k^{(h)} \alpha_{B \left( x^{(h)} ; r_k^{(h)} \right)} \left( \operatorname{Avg}_{F_k} f\right) - t_k^{(h)} \int f \mathrm{d} \theta^{(h)} \right| + \left| \left( t_k^{(h)} - \lambda_{i(k)}^{(h)} \right) \int f \mathrm{d} \theta^{(h)} \right| \right] \\
			\leq	& \left[ \sum_{h = 1}^n t_k^{(h)} \left| \alpha_{B \left( x^{(h)} ; r_k^{(h)} \right)} \left( \operatorname{Avg}_{F_k} f\right) - \int f \mathrm{d} \theta^{(h)} \right| \right] + \frac{\| f\|_{C(X)}}{k} .
		\end{align*}
		We can then estimate
		\begin{align*}
			& \left| \alpha_{B \left( x^{(h)} ; r_k^{(h)} \right)} \left( \operatorname{Avg}_{F_k} f\right) - \int f \mathrm{d} \theta^{(h)} \right| \\
			\leq	& \left| \alpha_{B \left( x^{(h)} ; r_k^{(h)} \right)} \left( \operatorname{Avg}_{F_k} f \right) - \operatorname{Avg}_{F_k} f \left( x^{(h)} \right) \right| + \left| \operatorname{Avg}_{F_k} f \left( x^{(h)} \right) - \int f \mathrm{d} \theta^{(h)} \right|
		\end{align*}
		Since $r_k^{(h)} < \delta_k$ for all $k \in \mathbb{N}$, it follows that if $\rho \left( x^{(h)} , y \right) < r_k^{(h)}$, then $\rho \left( T_g x^{(h)} , T_g y \right) < 1 / k$ for $g \in F_k$. Since $f$ is $1$-Lipschitz, it follows that $\left| f\left( T_g x^{(h)} \right) - f \left( T_g y \right) \right| < 1 / k$ for all $g \in F_k$. Thus
		\begin{align*}
			& \left| \alpha_{B \left( x^{(h)} ; r_k^{(h)} \right)} \left( \operatorname{Avg}_{F_k} f \right) - \operatorname{Avg}_{F_k} f \left( x^{(h)} \right) \right| \\
			=	& \left| \frac{1}{\mu \left( B \left( x^{(h)} ; r_k^{(h)} \right) \right)} \int_{B \left( x^{(h)} ; r_k^{(h)} \right)} \frac{1}{|F_k|} \sum_{g \in F_k} \left( f(T_g y) - f \left( T_g x^{(h)} \right) \right) \mathrm{d} \mu(y) \right| \\
			\leq	& \frac{1}{\mu \left( B \left( x^{(h)} ; r_k^{(h)} \right) \right)} \int_{B \left( x^{(h)} ; r_k^{(h)} \right)} \frac{1}{|F_k|} \sum_{g \in F_k} \left| f(T_g y) - f \left( T_g x^{(h)} \right) \right| \mathrm{d} \mu(y) \\
			<	& \frac{1}{\mu \left( B \left( x^{(h)} ; r_k^{(h)} \right) \right)} \int_{B \left( x^{(h)} ; r_k^{(h)} \right)} \frac{1}{|F_k|} \sum_{g \in F_k} \frac{1}{k} \mathrm{d} \mu(y) \\
			=	& \frac{1}{k} .
		\end{align*}
		Therefore
		\begin{align*}
			& \left| \alpha_{B \left( \bar{x}, \bar{r}_k \right)} \left( \operatorname{Avg}_{F_k} f \right) - \int f \mathrm{d} \nu_{i(k)} \right| \\
			\leq	& \left[ \sum_{h = 1}^n t_k^{(h)} \left| \alpha_{B \left( x^{(h)} ; r_k^{(h)} \right)} \left( \operatorname{Avg}_{F_k} f\right) - \int f \mathrm{d} \theta^{(h)} \right| \right] + \frac{\| f\|_{C(X)}}{k} \\
			\leq	& \left[ \sum_{h = 1}^n t_k^{(h)} \left( \left| \alpha_{B \left( x^{(h)} ; r_k^{(h)} \right)} \left( \operatorname{Avg}_{F_k} f \right) - \operatorname{Avg}_{F_k} f \left( x^{(h)} \right) \right| + \left| \operatorname{Avg}_{F_k} f \left( x^{(h)} \right) - \int f \mathrm{d} \theta^{(h)} \right| \right) \right] \\
				& + \frac{\| f\|_{C(X)}}{k} \\
			=	& \left[ \sum_{h = 1}^n t_k^{(h)} \left( \frac{1}{k} + \left| \operatorname{Avg}_{F_k} f \left( x^{(h)} \right) - \int f \mathrm{d} \theta^{(h)} \right| \right) \right] + \frac{\| f\|_{C(X)}}{k} \\
			=	& \frac{1}{k} + \left[ \sum_{h = 1}^n t_k^{(h)} \left| \operatorname{Avg}_{F_k} f \left( x^{(h)} \right) - \int f \mathrm{d} \theta^{(h)} \right| \right] + \frac{\|f\|_{C(X)}}{k} .
		\end{align*}
		
		Let $\left\{ f_m : m \in \mathbb{N} \right\}$ be a countable family of functions in $\mathcal{L}$ that densely span $C(X)$, and let $\operatorname{dist} : \mathcal{M}(X) \times \mathcal{M}(X) \to [0, 1]$ be the metric
		$$\operatorname{dist}(\beta_1, \beta_2) = \sum_{m = 1}^\infty 2^{-m} \min \left\{ \left| \int f_m \mathrm{d} (\beta_1 - \beta_2) \right| , 1 \right\} .$$
		This $\operatorname{dist}$ metric is compatible with the weak*-topology on $\mathcal{M}(X)$. We can also say that for all $M \in \mathbb{N}$, we have
		\begin{align*}
			& \operatorname{dist} \left( \alpha_{B \left( \bar{x} , \bar{r}_k \right)} \circ \operatorname{Avg}_{F_k} , \nu_{i(k)} \right) \\
			\leq	& \left[ \sum_{m = 1}^M 2^{-m} \left| \alpha_{B \left( \bar{x}, \bar{r}_k \right)} \left( \operatorname{Avg}_{F_k} f_m \right) - \int f_m \mathrm{d} \nu_{i(k)} \right| \right] + \sum_{m = M + 1}^\infty 2^{-m} \\
			\leq	& \left[ \sum_{m = 1}^M 2^{-m} \left[ \frac{1}{k} + \left[ \sum_{h = 1}^n t_k^{(h)} \left| \operatorname{Avg}_{F_k} f_m \left( x^{(h)} \right) - \int f_m \mathrm{d} \theta^{(h)} \right| \right] + \frac{\|f_m\|_{C(X)}}{k} \right] \right] + 2^{-M} \\
			\leq	& \frac{1 + \max_{1 \leq m \leq M} \|f_m\|_{C(X)}}{k} + 2^{-M} + \max_{1 \leq m \leq M} \max_{1 \leq h \leq n} \left| \operatorname{Avg}_{F_k} f_m \left( x^{(h)} \right) - \int f_m \mathrm{d} \theta^{(h)} \right|
		\end{align*}
		
		We claim that $\operatorname{LS} \left( \left( \alpha_{B \left( \bar{x}, \bar{r}_k \right)} \circ \operatorname{Avg}_{F_k} \right)_{k = 1}^\infty \right) = \mathcal{C}$.
		
		First, let $\nu \in \mathcal{C}$. Choose a sequence $\left( \nu_{i_\ell} \right)_{\ell = 1}^\infty$ such that $\operatorname{dist} \left( \nu, \nu_{i_\ell} \right) < 1 / \ell$. Choose $k_1 < k_2 < \cdots$ such that
		\begin{align*}
			k \geq k_\ell	& \Rightarrow \left| \operatorname{Avg}_{F_k} f_m \left( x^{(h)} \right) - \int f_m \mathrm{d} \theta^{(h)} \right| \leq \frac{1}{\ell}	& (m = 1, \ldots, \ell; h = 1, \ldots, n) , \\
			k_\ell	& \geq \ell \left( 1 + \max_{1 \leq m \leq \ell} \|f_m\|_{C(X)} \right) , \\
			k_\ell	& \in \mathcal{N}_{i_\ell}
		\end{align*}
		for all $\ell \in \mathbb{N}$. Then
		\begin{align*}
			& \operatorname{dist} \left( \alpha_{B \left( \bar{x} , \bar{r}_{k_\ell} \right)} \circ \operatorname{Avg}_{F_{k_\ell}} , \nu \right) \\
			\leq	& \operatorname{dist} \left( \alpha_{B \left( \bar{x} , \bar{r}_{k_\ell} \right)} \circ \operatorname{Avg}_{F_{k_\ell}} , \nu_{i \left( k_\ell \right)} \right) + \operatorname{dist} \left( \nu_{i_\ell} , \nu \right) \\
			\leq	& \left[ \frac{1 + \max_{1 \leq m \leq \ell} \|f_m\|_{C(X)}}{k_\ell} + 2^{-M} + \max_{1 \leq m \leq \ell} \max_{1 \leq h \leq n} \left| \operatorname{Avg}_{F_{k_\ell}} f_m \left( x^{(h)} \right) - \int f_m \mathrm{d} \theta^{(h)} \right| \right] \\
				& + \frac{1}{\ell} \\
			\leq	& \frac{1}{\ell} + 2^{-\ell} + \frac{1}{\ell} + \frac{1}{\ell} \\
			\stackrel{\ell \to \infty}{\to}	& 0 .
		\end{align*}
		Therefore $\nu \in \operatorname{LS} \left( \left( \alpha_{B \left( \bar{x}, \bar{r}_k \right)} \circ \operatorname{Avg}_{F_k} \right)_{k = 1}^\infty \right)$, meaning that $\mathcal{C} \subseteq \operatorname{LS} \left( \left( \alpha_{B \left( \bar{x}, \bar{r}_k \right)} \circ \operatorname{Avg}_{F_k} \right)_{k = 1}^\infty \right)$.
		
		To prove the opposite containment, suppose $\gamma \in \operatorname{LS} \left( \left( \alpha_{B \left( \bar{x}, \bar{r}_k \right)} \circ \operatorname{Avg}_{F_k} \right)_{k = 1}^\infty \right)$, and let $k_1 < k_2 < \cdots$ such that $\gamma = \lim_{\ell \to \infty} \alpha_{B \left( \bar{x}, \bar{r}_{k_\ell} \right)} \circ \operatorname{Avg}_{F_{k_\ell}}$. Fix $f \in \mathcal{L}$. Then
		\begin{align*}
			& \left| \int f \mathrm{d} \gamma - \int f \mathrm{d} \nu_{i \left( k_\ell \right)} \right| \\
			\leq	& \left| \int f \mathrm{d} \gamma - \alpha_{B \left( \bar{x}, \bar{r}_{k_\ell} \right)} \left( \operatorname{Avg}_{F_{k_\ell}} f \right) \right| + \left| \alpha_{B \left( \bar{x}, \bar{r}_{k_\ell} \right)} \left( \operatorname{Avg}_{F_{k_\ell}} f \right) - \int f \mathrm{d} \nu_{i \left( k_\ell \right)} \right| \\
			\leq	& \left| \int f \mathrm{d} \gamma - \alpha_{B \left( \bar{x}, \bar{r}_{k_\ell} \right)} \left( \operatorname{Avg}_{F_{k_\ell}} f \right) \right| + \frac{1}{k_\ell} \\
				& + \left[ \sum_{h = 1}^n t_{k_\ell}^{(h)} \left| \operatorname{Avg}_{F_{k_\ell}} f \left( x^{(h)} \right) - \int f \mathrm{d} \theta^{(h)} \right| \right] + \frac{\|f\|_{C(X)}}{k_\ell} \\
			\leq	& \left| \int f \mathrm{d} \gamma - \alpha_{B \left( \bar{x}, \bar{r}_{k_\ell} \right)} \left( \operatorname{Avg}_{F_{k_\ell}} f \right) \right| + \frac{1}{k_\ell} \\
				& + \left[ \max_{1 \leq h \leq n} \left| \operatorname{Avg}_{F_{k_\ell}} f \left( x^{(h)} \right) - \int f \mathrm{d} \theta^{(h)} \right| \right] + \frac{\|f\|_{C(X)}}{k_\ell} \\
			\stackrel{\ell \to \infty}{\to}	& 0 .
		\end{align*}
		Therefore $\gamma = \lim_{\ell \to \infty} \nu_{i\left( k_\ell \right)}$, meaning that $\gamma \in \mathcal{C}$. Thus $\operatorname{LS} \left( \left( \alpha_{B \left( \bar{x}, \bar{r}_k \right)} \circ \operatorname{Avg}_{F_k} \right)_{k = 1}^\infty \right) \subseteq \mathcal{C}$.
	\end{proof}
	
	In Theorem \ref{Targeting finite-dimensional subsets}, our assumption that $\mathcal{C}$ live in a finite-dimensional subset of $\mathcal{M}_T(X)$ helps us place an upper bound on $\operatorname{LS} \left( \alpha_{B(\bar{x}; \bar{r}_k)} \circ \operatorname{Avg}_{F_k} \right)_{k = 1}^\infty$, i.e. show that $\operatorname{LS} \left( \alpha_{B(\bar{x}; \bar{r}_k)} \circ \operatorname{Avg}_{F_k} \right)_{k = 1}^\infty \subseteq \mathcal{C}$. In general, it is possible to construct $(C_k)_{k = 1}^\infty$ for which $\operatorname{LS} \left( \alpha_{B(\bar{x}; \bar{r}_k)} \circ \operatorname{Avg}_{F_k} \right)_{k = 1}^\infty$ is ``maximally large," as the following theorem shows.
	
	\begin{Thm}\label{Go big or go home}
		Suppose $\mu$ is a Borel probability measure on $X$. Then there exists a sequence $(C_k)_{k = 1}^\infty$ of multi-balls in $X$ such that
		$$
		\operatorname{LS} \left( \left( \alpha_{C_k} \circ \operatorname{Avg}_{F_k} \right)_{k = 1}^\infty \right) = \mathcal{M}_T(X) .
		$$
	\end{Thm}
	
	\begin{proof}
		Since $\operatorname{LS} \left( \left( \alpha_{C_k} \circ \operatorname{Avg}_{F_k} \right)_{k = 1}^\infty \right)$ is always a closed subset of $\mathcal{M}_T(X)$, it will suffice to construct $(C_k)_{k = 1}^\infty$ such that $\operatorname{LS} \left( \left( \alpha_{C_k} \circ \operatorname{Avg}_{F_k} \right)_{k = 1}^\infty \right)$ is dense in $\mathcal{M}_T(X)$.
		
		Let $\mathcal{E} = \left\{ \theta^{(h)} : h \in \mathbb{N} \right\} \subseteq \partial_e \mathcal{M}_T(X)$ be a countable dense subset of $\partial_e \mathcal{M}_T(X)$, and set
		$$\mathcal{F} = \left\{ \sum_{h = 1}^{n} \lambda^{(h)} \theta^{(h)} : n \in \mathbb{N} , \bar{\lambda} \in [0, 1]^n \cap \mathbb{Q}^n, \sum_{h = 1}^n \lambda^{(h)} = 1 \right\} ,$$
		i.e. $\mathcal{F}$ is the set of all rational convex combinations of elements of $\mathcal{E}$. Assume that the $\theta^{(h)} , h \in \mathbb{N}$ are distinct. By the Krein-Millman Theorem, the set $\mathcal{F}$ is a countable dense subset of $\mathcal{M}_T(X)$. Let $\left\{ \nu_i : i \in I \right\}$ be an enumeration of $\mathcal{F}$, where $I$ is some countable indexing set, and let $\mathscr{N} = \left\{ \mathcal{N}_i : i \in I \right\}$ be a partition of $\mathbb{N}$ into countably infinitely many infinite subsets.
		
		For each $i \in I$, let $\left( \kappa(i, \ell) \right)_{\ell = 1}^\infty$ be a strictly increasing sequence such that
		\begin{align*}
			\kappa(i, \ell)	& \in \mathcal{N}_i , \\
			\left( F_{\kappa(i, \ell)} \right)_{\ell = 1}^\infty	& \textrm{is tempered},
		\end{align*}
		which exists by Lemma \ref{Tempered subsequence}.
		
		We are going to construct $(C_k)_{k = 1}^\infty$ such that $\lim_{\ell \to \infty} \alpha_{C_{\kappa(i, \ell)}} \circ \operatorname{Avg}_{F_{\kappa(i, \ell)}} = \nu_i$ for all $i \in I$. For each $k \in \mathbb{N}$, set $i(k) \in I$ such that $k \in \mathcal{N}_{i(k)}$.
		
		For each $i \in I$, choose $\bar{\lambda}_i \in \left( [0, 1] \cap \mathbb{Q} \right)^\mathbb{N}$ and $n_i \in \mathbb{N}$ such that
		\begin{align*}
			\sum_{h = 1}^{n_i} \lambda_i^{(h)} \theta^{(h)}	& = \nu_i , \\
			\sum_{h = 1}^{n_i} \lambda_i^{(h)}	& = 1 , \\
			\lambda_{i}^{(h)}	& = 0 & \textrm{for all $h > n_i$}.
		\end{align*}
		By the Lindenstrauss Ergodic Theorem, there exists for each $\theta^{(h)}$ a point $x^{(h)} \in X$ such that
		\begin{align*}
			\lim_{\ell \to \infty} \operatorname{Avg}_{F_{\kappa(i, \ell)}} f \left( x^{(h)} \right)	& = \int f \mathrm{d} \theta^{(h)}	& (\forall f \in C(X) , \; \forall i \in I) .
		\end{align*}
		
		For each $k \in \mathbb{N}$, choose $\bar{t}_k = \left( t_k^{(1)} , \ldots, t_k^{\left( n_{i(k)} \right)} \right) \in (0, 1)^{n_{i(k)}}$ such that
		\begin{align*}
			\sum_{h = 1}^{n_{i(k)}} \left| t_k^{(h)} - \lambda_{i(k)}^{(h)} \right|	& < 1 / k , \\
			\sum_{h = 1}^{n_{i(k)}} t_k^{(h)}	& = 1 .
		\end{align*}
		For each $k \in \mathbb{N}$, choose $\delta_k > 0$ such that
		$$
		\max_{g \in F_k} \left( L(g) \cdot \delta_k^{H(g)} \right) < 1 / k .
		$$
		Now for each $k \in \mathbb{N}$, use Lemma \ref{Controlling relative measures} to choose $r_k^{(1)} , \ldots, r_k^{\left(n_{i(k)} \right)} \in (0, 1)$ such that
		\begin{align*}
			t_k^{(j)}	& = \frac{\mu \left( B \left( x^{(h)} ; r_k^{(h)} \right) \right)}{ \mu \left( B \left( x^{(1)} ; r_k^{(1)} \right) \right) + \cdots + \mu \left( B \left( x^{\left( n_{i(k)} \right)} ; r_k^{\left(n_{i(k)}\right)} \right) \right)} , \\
			r_k^{(h)}	& < \delta_k , \\
			r_k^{(h)}	& < \frac{1}{3} \min \left\{ \rho \left( x^{(h_1)} , x^{(h_2)} \right) : 1 \leq h_1 < h_2 \leq n_{i(k)} \right\} .
		\end{align*}
		The last condition ensures that the balls $\left\{ B \left( x^{(h)} ; r_k^{(h)} \right) : h = 1, \ldots, n_{i(k)} \right\}$ are pairwise disjoint. Since the points $x^{(h)}$ each satisfy
		$$
		\lim_{\ell \to \infty} \operatorname{Avg}_{F_{\kappa(i, \ell)}} f \left( x^{(h)} \right) = \int f \mathrm{d} \theta^{(h)}
		$$
		for all $f \in C(X) , i \in I$, and the measures $\theta^{(h)}$ are distinct, it follows that the $x^{(h)}$ are also distinct, meaning that $\min \left\{ \rho \left( x^{(h_1)} , x^{(h_2)} \right) : 1 \leq h_1 < h_2 \leq n_{i(k)} \right\} > 0$.
		
		For each $k \in \mathbb{N}$, set
		$$C_k = B \left( x^{(1)} , \ldots , x^{\left( n_{i(k)} \right)} ; r_k^{(1)} , \ldots, r_k^{\left( n_{i(k)} \right)} \right) .$$
		We now show that
		\begin{align*}
			\lim_{\ell \to \infty} \alpha_{ C_{\kappa(i, \ell)} } \left( \operatorname{Avg}_{F_{\kappa(i, \ell)}} f \right)	& = \int f \mathrm{d} \nu_i	& (\forall f \in C(X) , \; \forall i \in I) .
		\end{align*}
		In light of Lemma \ref{1-Lipschitz functions dense span}, it will suffice to prove the convergence for $f \in \mathcal{L}$, where
		$$\mathcal{L} = \left\{ \phi \in C(X) : \forall x \in X \; \forall y \in X \; \left( \left| \phi(x) - \phi(y) \right| \leq \rho(x, y) \right) \right\} $$
		is the family of all $1$-Lipschitz functions. We see
		\begin{align*}
			& \left| \alpha_{C_k} \left( \operatorname{Avg}_{F_k} f \right) - \int f \mathrm{d} \nu_{i(k)} \right| \\
			[\textrm{Lem. \ref{Phasing for multi-balls}}]=	& \left| \left[ \sum_{h = 1}^{n_{i(k)}} \frac{\mu\left( B \left( x^{(h)} ; r_k^{(h)} \right) \right)}{\sum_{u = 1}^n \mu \left( B \left( x^{(u)} ; r^{(u)} \right) \right)} \alpha_{B \left( x^{(h)} ; r_k^{(h)} \right)} \left( \operatorname{Avg}_{F_k} f \right) \right] - \int f \mathrm{d} \nu_{i(k)} \right| \\
			=	& \left| \left[ \sum_{h = 1}^{n_{i(k)}} t_k^{(h)} \alpha_{B \left( x^{(h)} ; r_k^{(h)} \right)} \left( \operatorname{Avg}_{F_k} f \right) \right] - \int f \mathrm{d} \nu_{i(k)} \right| \\
			=	& \left| \sum_{h = 1}^{n_{i(k)}} \left( t_k^{(h)} \alpha_{B \left( x^{(h)} ; r_k^{(h)} \right)} \left( \operatorname{Avg}_{F_k} f\right) - \lambda_{i(k)}^{(h)} \int f \mathrm{d} \theta^{(h)} \right) \right| \\
			\leq	& \sum_{h = 1}^{n_{i(k)}} \left| t_k^{(h)} \alpha_{B \left( x^{(h)} ; r_k^{(h)} \right)} \left( \operatorname{Avg}_{F_k} f\right) - \lambda_{i(k)}^{(h)} \int f \mathrm{d} \theta^{(h)} \right| \\
			\leq	& \sum_{h = 1}^{n_{i(k)}} \left[ \left| t_k^{(h)} \alpha_{B \left( x^{(h)} ; r_k^{(h)} \right)} \left( \operatorname{Avg}_{F_k} f\right) - t_k^{(h)} \int f \mathrm{d} \theta^{(h)} \right| + \left| \left( t_k^{(h)} - \lambda_{i(k)}^{(h)} \right) \int f \mathrm{d} \theta^{(h)} \right| \right] \\
			\leq	& \left[ \sum_{h = 1}^{n_{i(k)}} t_k^{(h)} \left| \alpha_{B \left( x^{(h)} ; r_k^{(h)} \right)} \left( \operatorname{Avg}_{F_k} f\right) - \int f \mathrm{d} \theta^{(h)} \right| \right] + \frac{\| f \|_{C(X)}}{k} .
		\end{align*}
		We can then estimate
		\begin{align*}
			& \left| \alpha_{B \left( x^{(h)} ; r_k^{(h)} \right)} \left( \operatorname{Avg}_{F_k} f\right) - \int f \mathrm{d} \theta^{(h)} \right| \\
			\leq	& \left| \alpha_{B \left( x^{(h)} ; r_k^{(h)} \right)} \left( \operatorname{Avg}_{F_k} f \right) - \operatorname{Avg}_{F_k} f \left( x^{(h)} \right) \right| + \left| \operatorname{Avg}_{F_k} f \left( x^{(h)} \right) - \int f \mathrm{d} \theta^{(h)} \right|
		\end{align*}
		Since $r_k^{(h)} < \delta_k$ for all $k \in \mathbb{N}$, it follows that if $\rho \left( x^{(h)} , y \right) < r_k^{(h)}$, then $\rho \left( T_g x^{(h)} , T_g y \right) < 1 / k$ for $g \in F_k$. Since $f$ is $1$-Lipschitz, it follows that $\left| f\left( T_g x^{(h)} \right) - f \left( T_g y \right) \right| < 1 / k$ for all $g \in F_k$. Thus
		\begin{align*}
			& \left| \alpha_{B \left( x^{(h)} ; r_k^{(h)} \right)} \left( \operatorname{Avg}_{F_k} f \right) - \operatorname{Avg}_{F_k} f \left( x^{(h)} \right) \right| \\
			=	& \left| \frac{1}{\mu \left( B \left( x^{(h)} ; r_k^{(h)} \right) \right)} \int_{B \left( x^{(h)} ; r_k^{(h)} \right)} \frac{1}{|F_k|} \sum_{g \in F_k} \left( f(T_g y) - f \left( T_g x^{(h)} \right) \right) \mathrm{d} \mu(y) \right| \\
			\leq	& \frac{1}{\mu \left( B \left( x^{(h)} ; r_k^{(h)} \right) \right)} \int_{B \left( x^{(h)} ; r_k^{(h)} \right)} \frac{1}{|F_k|} \sum_{g \in F_k} \left| f(T_g y) - f \left( T_g x^{(h)} \right) \right| \mathrm{d} \mu(y) \\
			<	& \frac{1}{\mu \left( B \left( x^{(h)} ; r_k^{(h)} \right) \right)} \int_{B \left( x^{(h)} ; r_k^{(h)} \right)} \frac{1}{|F_k|} \sum_{g \in F_k} \frac{1}{k} \mathrm{d} \mu(y) \\
			=	& \frac{1}{k} .
		\end{align*}
		Therefore
		\begin{align*}
			& \left| \alpha_{C_k} \left( \operatorname{Avg}_{F_k} f \right) - \int f \mathrm{d} \nu_{i(k)} \right| \\
			\leq	& \left[ \sum_{h = 1}^{n_{i(k)}} t_k^{(h)} \left| \alpha_{B \left( x^{(h)} ; r_k^{(h)} \right)} \left( \operatorname{Avg}_{F_k} f\right) - \int f \mathrm{d} \theta^{(h)} \right| \right] + \frac{\| f\|_{C(X)}}{k} \\
			\leq	& \left[ \sum_{h = 1}^{n_{i(k)}} t_k^{(h)} \left( \left| \alpha_{B \left( x^{(h)} ; r_k^{(h)} \right)} \left( \operatorname{Avg}_{F_k} f \right) - \operatorname{Avg}_{F_k} f \left( x^{(h)} \right) \right| + \left| \operatorname{Avg}_{F_k} f \left( x^{(h)} \right) - \int f \mathrm{d} \theta^{(h)} \right| \right) \right] \\
				& + \frac{\| f\|_{C(X)}}{k} \\
			=	& \left[ \sum_{h = 1}^{n_{i(k)}} t_k^{(h)} \left( \frac{1}{k} + \left| \operatorname{Avg}_{F_k} f \left( x^{(h)} \right) - \int f \mathrm{d} \theta^{(h)} \right| \right) \right] + \frac{\| f\|_{C(X)}}{k} \\
			=	& \frac{1}{k} + \left[ \sum_{h = 1}^{n_{i(k)}} t_k^{(h)} \left| \operatorname{Avg}_{F_k} f \left( x^{(h)} \right) - \int f \mathrm{d} \theta^{(h)} \right| \right] + \frac{\|f\|_{C(X)}}{k} .
		\end{align*}
		
		In particular, this tells us that for fixed $i \in I$, we have
		\begin{align*}
			& \left| \alpha_{C_{\kappa(i, \ell)}} \left( \operatorname{Avg}_{F_{\kappa(i, \ell)}} f \right) - \int f \mathrm{d} \nu_{i(k)} \right| \\
			=	& \left| \alpha_{C_{\kappa(i, \ell)}} \left( \operatorname{Avg}_{F_{\kappa(i, \ell)}} f \right) - \int f \mathrm{d} \nu_{i(\kappa(i, \ell))} \right| \\
			\leq	& \frac{1}{\kappa(i, \ell)} + \left[ \sum_{h = 1}^{n_{i}} t_{\kappa(i, \ell)}^{(h)} \left| \operatorname{Avg}_{F_{\kappa(i, \ell)}} f \left( x^{(h)} \right) - \int f \mathrm{d} \theta^{(h)} \right| \right] + \frac{\|f\|_{C(X)}}{\kappa(i, \ell)} \\
			\leq	& \frac{1}{\ell} + \left[ \max_{1 \leq h \leq n_i} \left| \operatorname{Avg}_{F_{\kappa(i, \ell)}} f \left( x^{(h)} \right) - \int f \mathrm{d} \theta^{(h)} \right| \right] + \frac{\|f\|_{C(X)}}{\ell} \\
			\stackrel{\ell \to \infty}{\to}	& 0 .
		\end{align*}
		Therefore $\nu_i = \lim_{\ell \to \infty} \alpha_{ C_{\kappa(i, \ell)}} \operatorname{Avg}_{F_{\kappa(i, \ell)}}$. Thus $\operatorname{LS} \left( \left( \alpha_{C_k} \circ \operatorname{Avg}_{F_k} \right)_{k = 1}^\infty \right) \supseteq \mathcal{F}$ is dense in $\mathcal{M}_T(X)$, and since $\operatorname{LS} \left( \left( \alpha_{C_k} \circ \operatorname{Avg}_{F_k} \right)_{k = 1}^\infty \right)$ is a closed subset of $\mathcal{M}_T(X)$, it follows that
		$$\operatorname{LS} \left( \left( \alpha_{C_k} \circ \operatorname{Avg}_{F_k} \right)_{k = 1}^\infty \right) = \mathcal{M}_T(X) .$$
	\end{proof}
	
	We conclude this section by proving a result that does not rely on the measure $\mu$ neglecting shells.
	
	\begin{Prop}
		There exists a sequence $(x_k)_{k = 1}^\infty$ of points in $X$ and a sequence $(r_k)_{k = 1}^\infty$ of radii such that
		$$
		\operatorname{LS} \left( \left( \alpha_{B(x_k; r_k)} \circ \operatorname{Avg}_{F_k} \right)_{k = 1}^\infty \right) \supseteq \partial_e \mathcal{M}_T(X) .
		$$
	\end{Prop}
	
	\begin{proof}
		Let $\left\{ \nu_i : i \in I \right\}$ be a countable dense subset of $\partial_e \mathcal{M}_T(X)$, where $I$ is some countable indexing set, and let $\mathscr{N} = \left\{ \mathcal{N}_i : i \in I \right\}$ be a partition of $\mathbb{N}$ into countably infinitely many infinite subsets. For each $k \in \mathbb{N}$, set $i(k) \in I$ such that $k \in \mathcal{N}_{i(k)}$.
		
		For each $i \in I$, let $\left( \kappa(i, \ell) \right)_{\ell = 1}^\infty$ be a strictly increasing sequence such that
		\begin{align*}
			\kappa(i, \ell)	& \in \mathcal{N}_i , \\
			\left( F_{\kappa(i, \ell)} \right)_{\ell = 1}^\infty	& \textrm{is tempered},
		\end{align*}
		which exists by Lemma \ref{Tempered subsequence}. By the Lindenstrauss Ergodic Theorem, for each $i \in I$ exists $y_i \in X$ such that
		\begin{align*}
			\lim_{\ell \to \infty} \operatorname{Avg}_{F_{\kappa(i, \ell)}} f(y_i)	& = \int f \mathrm{d} \nu_i	& (\forall f \in C(X)) .
		\end{align*}
		Set $x_k = y_{i(k)}$.
		
		For each $k \in \mathbb{N}$, choose $\delta_k > 0$ such that
		$$
		\max_{g \in F_k} \left( L(g) \cdot \delta_k^{H(g)} \right) < 1 / k ,
		$$
		and let $r_k \in (0, \delta_k)$ for all $k \in \mathbb{N}$. If $f \in \mathcal{L}(X)$, then
		\begin{align*}
			& \left| \alpha_{B\left( x_{\kappa(i, \ell)}; r_{\kappa(i, \ell)} \right)} \left( \operatorname{Avg}_{\kappa(i, \ell)} f \right) - \int f \mathrm{d} \nu_i \right| \\
			\leq	&  \left| \alpha_{B\left( x_{\kappa(i, \ell)}; r_{\kappa(i, \ell)} \right)} \left( \operatorname{Avg}_{\kappa(i, \ell)} f \right) - \operatorname{Avg}_{\kappa(i, \ell)} f\left( x_{\kappa(i, \ell)} \right) \right| + \left| \operatorname{Avg}_{\kappa(i, \ell)} f\left( x_{\kappa(i, \ell)} \right) - \int f \mathrm{d} \nu_i \right| \\
			=	& \left| \alpha_{B\left( x_{\kappa(i, \ell)}; r_{\kappa(i, \ell)} \right)} \left( \operatorname{Avg}_{\kappa(i, \ell)} f \right) - \operatorname{Avg}_{\kappa(i, \ell)} f\left( x_{\kappa(i, \ell)} \right) \right| + \left| \operatorname{Avg}_{\kappa(i, \ell)} f\left( y_i \right) - \int f \mathrm{d} \nu_i \right| \\
			\leq	& \frac{1}{\kappa(i, \ell)} + \left| \operatorname{Avg}_{\kappa(i, \ell)} f\left( y_i \right) - \int f \mathrm{d} \nu_i \right| \\
			\stackrel{\ell \to \infty}{\to}	& 0 .
		\end{align*}
		
		Therefore $\nu_i \in \operatorname{LS} \left( \left( \alpha_{B(x_k; r_k)} \circ \operatorname{Avg}_{F_k} \right)_{k = 1}^\infty \right)$ for all $i \in I$. Since $\left\{ \nu_i : i \in I \right\}$ is dense in $\partial_e \mathcal{M}_T(X)$, and $\operatorname{LS} \left( \left( \alpha_{B(x_k; r_k)} \circ \operatorname{Avg}_{F_k} \right)_{k = 1}^\infty \right)$ is always closed, it follows that
		$$\partial_e \mathcal{M}_T(X) \subseteq \operatorname{LS} \left( \left( \alpha_{B(x_k; r_k)} \circ \operatorname{Avg}_{F_k} \right)_{k = 1}^\infty \right) .$$
	\end{proof}
	
	\section{Weak specification and maximal oscillation}\label{Specification}
	
	Specification properties were initially introduced by R. Bowen in \cite{BowenSpec} in the course of studying Axiom A diffeomorphisms. In the intervening decades, a considerable amount of effort has been put into the study of other specification-like properties -typically weaker than the Specification Property considered by Bowen- and the connections between them. For a broad overview of these specification-like properties and the relations between them, we refer the reader to \cite{Panorama}, whose terminology we will be following.
	
	Throughout this section, let $(X, \rho)$ be a compact metric space, and let $T : \mathbb{N}_0 \curvearrowright X$ be an action of $\mathbb{N}_0$ on $X$ by continuous (not necessarily invertible) maps. For $x \in X , k \in \mathbb{N}$, we define the \emph{$k$th empirical measure of $x$} to be the Borel probability measure
	$$\mu_{x, k} : = \sum_{j = 0}^{k - 1} \delta_{T_j x} ,$$
	where $\delta_y$ denotes the point mass at $y$, i.e. $\delta_y(A) = \chi_A(y)$. In light of Lemma \ref{Singly local pointwise reduction}, the study of local temporo-spatial differentiations is closely tied to the study of pointwise ergodic averages.
	
	A point $x \in X$ is said to have \emph{maximal oscillation} with respect to $T : \mathbb{N}_0 \curvearrowright X$ if
	$$
	\operatorname{LS} \left( \left( \mu_{x, k} \right)_{k = 1}^\infty \right) = \mathcal{M}_T(X) .
	$$
	This could be understood as the worst possible divergence for the sequence $\left( \mu_{x, k} \right)_{k = 1}^\infty$. M. Denker, C. Grillenberger, and	K. Sigmund demonstrated the following prevalence result for points of maximal oscillation. Recall that a subset $S$ of $X$ is called \emph{residual} if $S$ contains a dense $G_\delta$ set.
	
	\begin{Thm}\label{DGS}\cite[Proposition 21.18]{DGSonCompactSpaces}
		If $T$ has the Periodic Specification Property, then the set of points $x \in X$ with maximal oscillation is residual in $X$.
	\end{Thm}
	
	\begin{Rmk}
		In \cite{DGSonCompactSpaces}, what the authors call the Specification Property (defined there as Definition 21.1) is what \cite{Panorama} calls the Periodic Specification Property, which is slightly stronger than what \cite{Panorama} -and consequently we- call the Specification Property in Definition \ref{Spec Properties}.
	\end{Rmk}
	
	We introduce here a variation on and strengthening of the definition of maximal oscillation.
	
	\begin{Def}
		A \emph{sampling family} is a family $\Pi$ of functions $\mathbb{N} \to \mathbb{N}$ such that $\lim_{k \to \infty} \pi(k) = + \infty$ for all $k \in \mathbb{N}$. Given a sampling family $\Pi$, we say that a point $x \in X$ has \emph{maximal oscillation relative to $\Pi$} if for every $\pi \in \Pi$, we have that
		$$\operatorname{LS} \left( \left( \mu_{x, \pi(k)} \right)_{k = 1}^\infty \right) = \mathcal{M}_T (X) .$$
	\end{Def}
	
	Maximal oscillation can then be recovered as the case where $\Pi = \left\{ k \mapsto k \right\}$ consists solely of the identity function on $\mathbb{N}$.
	
	Maximal oscillation describes the situation where not only does the sequence $\left( \mu_{x, k} \right)_{k = 1}^\infty$ diverge, but it diverges to the greatest extent possible. However, because $\left( \mu_{x, k} \right)_{k = 1}^\infty$ takes values in the compact space $\mathcal{M}(X)$, we know it will always have convergent subsequences, meaning this divergence will always ``disappear" if we restrict our attention to an appropriate subsequence. Our notion of maximal oscillation relative to a sampling family allows us to strengthen the notion of maximal oscillation by prescribing the ``worst-case scenario" divergence along a family of subsequences.
	
	We now define a hierarchy of specification-like properties.
	
	\begin{Def}\label{Defining Spec}
		A \emph{specification} is a finite sequence $\xi = \left\{ \left( [a_j, b_j] , x_j \right) \right\}_{j = 1}^n$ of finite subintervals $[a_j, b_j]$ of $\mathbb{N}$ and points $x_j \in X$. Given a function $\mathbf{M} : \mathbb{N} \to \mathbb{N}$, we say that the specification $\xi = \left\{ \left( [a_j, b_j] , x_j \right) \right\}_{j = 1}^n$ is \emph{$\mathbf{M}$-spaced} if $a_j - b_{j - 1} \geq \mathbf{M}(j)$ for all $j = 2, \ldots, n$. If $\mathbf{M}$ is the constant function $N \in \mathbb{N}$, then we say an $\mathbf{M}$-spaced specification is \emph{$N$-spaced}.
	\end{Def}
	
	\begin{Def}\label{Spec Properties}
		Let $\xi = \left\{ \left( [a_j, b_j] , x_j \right) \right\}_{j = 1}^n$ be a specification, and let $\delta > 0$. We call a point $y \in X$ a \emph{$\delta$-tracing} of $\xi$ if
		\begin{align*}
			\rho \left( T_i x_j , T_{a_j + i} y \right)	& < \delta	& \left( \forall j = 1, \ldots, n ; i = 0 , 1, \ldots, b_j - a_j \right) .
		\end{align*}
		\begin{enumerate}[label=(\Roman*)]
			\item We call a family of functions $\left( \mathbf{M}_\delta : \mathbb{N} \to \mathbb{N} \right)_{\delta \in (0, 1)}$ a \emph{modulus of specification} for $(X, T)$ if every $\mathbf{M}_\delta$-spaced specification $\xi$ admits a $\delta$-tracing, and say that $T : \mathbb{N}_0 \curvearrowright X$ has the \emph{Very Weak Specification Property}.
			\item If $T$ admits a modulus of specification $\left( \mathbf{M}_\delta \right)_{\delta \in (0, 1)}$ with the additional property that
			\begin{align*}
				\lim_{n \to \infty} \frac{\mathbf{M}_\delta(n)}{n}	& = 0	& (\forall \delta \in (0, 1)),
			\end{align*}
			then we say that $T$ has the \emph{Weak Specification Property}.
			\item If $T$ admits a modulus of specification $\left( \mathbf{M}_\delta \right)_{\delta \in (0, 1)}$ with the additional property that each $\mathbf{M}_\delta$ is a constant function, then we say that $T$ has the \emph{Specification Property}.
		\end{enumerate}
	\end{Def}
	
	Intuitively, these specification-like properties mean that if we have some orbit segments that we want to approximate within $\delta$, then we can find a point whose orbits are close to those segments as long as the segments are spaced far enough apart from each other. Clearly these specification properties are listed in ascending order of strength.
	
	What we call the Weak Specification Property and Specification Property both have precedents in the literature. The Specification Property goes back to R. Bowen's original work \cite{BowenSpec}, and what we call here the Weak Specification Property can be found in \cite{MarcusWeakSpec}. See \cite{Panorama} for a fuller historical discussion. However, to our knowledge, there is no precedent for what we term here the Very Weak Specification Property in the literature. Regardless, our results in this section do not rely on a modulus of specification $\left( \mathbf{M}_\delta \right)_{\delta \in (0, 1)}$ satisfying the condition that $\mathbf{M}_\delta(n) = o(n)$ for all $\delta \in (0, 1)$, so we see fit to introduce this weaker specification-like property.
	
	Our main theorem of this section is the following.
	
	\begin{Thm}\label{Weak Spec Theorem with sampling}
		Let $\Pi$ be a countable sampling family. Suppose $T : \mathbb{N}_0 \curvearrowright X$ has the Very Weak Specification Property. Then the set
		$$X^\Pi = \left\{ x \in X : \operatorname{LS} \left( \left( \mu_{x, \pi(k)} \right)_{k = 1}^\infty \right) = \mathcal{M}_T(X) \textrm{ for all $\pi \in \Pi$} \right\} .$$
		is residual.
	\end{Thm}
	
	Let $\mathcal{E}$ denote a countable dense subset of $\partial_e \mathcal{M}_T(X)$, and let
	$$\mathcal{F} = \left\{ \sum_{i = 1}^n \lambda_i \theta_i : n \in \mathbb{N}, \theta_i \in \mathcal{E} , \lambda_i \in \mathbb{Q} \cap [0, 1], \sum_{i = 1}^n \lambda_i = 1 \right\} ,$$
	i.e. $\mathcal{F}$ is the set of all rational convex combinations of elements of $\mathcal{E}$. Then $\mathcal{F}$ is a countable dense subset of $\mathcal{M}_T(X)$ by the Krein-Millman Theorem. Further, let $\{f_h\}_{h = 1}^\infty$ be an enumerated dense subset of $C(X)$.
	
	\begin{Lem}\label{Weak Spec Lemma with sampling}
		Let $\Pi$ be a sampling family. For $\nu \in \mathcal{F}, \epsilon > 0 , H \in \mathbb{N}, k_0 \in \mathbb{N}, \pi \in \Pi$, set
		$$E(\nu, \epsilon, H, k_0, \pi) = \bigcap_{h = 1}^H \left\{ x \in X : \exists k \geq k_0 \; \left( \left| \left( \frac{1}{\pi(k)} \sum_{j = 0}^{\pi(k) - 1} T_j f_h(x) \right) - \int f_h \mathrm{d} \nu \right| < \epsilon \right) \right\} .$$
		If $T$ has the Very Weak Specification Property, then $E(\nu, \epsilon, H, k_0, \pi)$ is a dense open subset of $X$.
	\end{Lem}
	
	\begin{proof}
		Fix $H \in \mathbb{N} , \nu \in \mathcal{M}_T(X), \epsilon > 0, \pi \in \Pi$. Set
		$$A_k = \left\{ x \in X : \left| \left( \frac{1}{\pi(k)} \sum_{j = 0}^{\pi(k) - 1} T_j f_h(x) \right) - \int f_h \mathrm{d} \nu \right| < \epsilon \textrm{ for $h = 1, \ldots, H$} \right\} .$$
		Then $E(\nu, \epsilon, H, k_0) = \bigcup_{k = k_0}^\infty A_k$. Clearly $\bigcup_{k = k_0}^\infty A_k$ is open, leaving us to show it is dense.
		
		Choose $\theta_0, \theta_1, \ldots, \theta_{I - 1} \in \mathcal{E}; \lambda_0, \lambda_1, \ldots, \lambda_{I - 1} \in [0, 1] \cap \mathbb{Q}$ such that
		$$\nu = \sum_{i = 0}^{I - 1} \lambda_i \theta_i ,$$
		where we can assume without loss of generality that $\lambda_i > 0$ for all $i = 1, \ldots, I$. Let \linebreak$p_0, p_1, \ldots, p_{I - 1}, q \in \mathbb{N}$ such that
		\begin{align*}
			\lambda_i	& = \frac{p_i}{q}	& (i = 0, 1, \ldots, I - 1) .
		\end{align*}
		Let $y_0, y_1, \ldots, y_{I - 1} \in X$ such that $\lim_{k \to \infty} \frac{1}{k} \sum_{j = 0}^{k - 1} T_j f_h(y_i) = \int f \mathrm{d} \theta_i$ for $i = 0, 1, \ldots, I - 1$, which exist by the Birkhoff Ergodic Theorem. Choose $k_0 \in \mathbb{N}$ such that
		\begin{align*}
			k \geq k_0	& \Rightarrow \left| \left( \frac{1}{k} \sum_{j = 0}^{k - 1} T_j f_h(y_i) \right) - \int f_h \mathrm{d} \theta_i \right| < \epsilon / 3	& (i = 0, 1, \ldots, I - 1; h = 1, \ldots, H) .
		\end{align*}
		
		Fix $x \in X, \eta > 0$. We will show that there exists $k \geq k_0$ and $y \in A_k$ such that
		$$\rho(x, y) \leq \eta .$$
		Since $f_1, \ldots, f_H$ are uniformly continuous, there exists $\delta > 0$ such that
		$$
		\forall z_1, z_2 \in X \; \forall h \in \{1, \ldots, H\} \;\left( \rho(z_1, z_2) < \delta \Rightarrow |f_h(z_1) - f_h(z_2)| < \epsilon / 3 \right) .
		$$
		Assume without loss of generality that $\delta < \eta$.
		
		Let $(\mathbf{M}_\delta)_{\delta \in (0, 1)}$ be a modulus of specification for $T : \mathbb{N}_0 \curvearrowright X$. Fix 
		$$N = \max \left\{ \mathbf{M}_\delta(1) , \ldots, \mathbf{M}_\delta(I + 1) \right\}.$$
		For $K \in \mathbb{N}$, define a sequence
		$$a_{-1}^{(K)} \leq b_{-1}^{(K)} < a_0^{(K)} \leq b_0^{(K)} < a_1^{(K)} \leq b_1^{(K)} < a_2^{(K)} \leq b_2^{(K)} < \cdots < a_{I - 1}^{(K)} \leq b_{I - 1}^{(K)}$$
		by
		\begin{align*}
			a_{- 1}^{(K)}	& = 0 ,	& b_{-1}^{(K)}	& = 0 , \\
			a_0^{(K)}	& = N ,	& b_0^{(K)}	& = a_0 + K p_0 - 1, \\
			a_1^{(K)}	& = b_0^{(K)} + N,	& b_1^{(K)}	& = a_1^{(K)} + K p_1 - 1 , \\
			a_2^{(K)}	& = b_1^{(K)} + N,	& b_2^{(K)}	& = a_2^{(K)} + K p_2 - 1 , \\
			\vdots \\
			a_{I - 1}^{(K)}	& = b_{I - 2}^{(K)} + N ,	& b_{I - 1}^{(K)}	& = a_{I - 1}^{(K)} + K p_{I - 1} - 1 .
		\end{align*}
		Written explicitly, we have
		\begin{align*}
			a_i^{(K)}	& = (i + 1) N + K \sum_{\ell = 0}^{i - 1} p_\ell , \\
			b_i^{(K)}	& = (i + 1) N - 1 + K \sum_{\ell = 0}^{i} p_\ell .
		\end{align*}
		Set
		$$x_i = \begin{cases}
			x	& \textrm{if $i = - 1$} , \\
			y_{i}	& \textrm{if $0 \leq i \leq I - 1$}
		\end{cases}
		$$
		Let $\xi^{(K)}$ be the specification
		$$\xi^{(K)} = \left\{ \left( \left[ a_i^{(K)}, b_i^{(K)} \right] , x_i \right) \right\}_{i = -1}^{I - 1} .$$
		Then $\xi^{(K)}$ is $\mathbf{M}_\delta$-spaced, so by the Weak Specification Property, for each $K \in \mathbb{N}$ exists $y = y^{(K)} \in X$ such that $y^{(K)}$ is a $\delta$-tracing of $\xi^{(K)}$. In particular, since $a_{-1}^{(K)} = 0 = b_{-1}^{(K)}, x_{-1} = x$, this means that $\rho(x, y) < \delta < \eta$. We claim that $y^{(K)} \in E(\nu, \epsilon, H, k_0, \pi)$ for sufficiently large $K$.
		
		For $k \in \mathbb{N}$, set
		$$K = K_k = \left \lfloor \frac{\pi(k) - IN - 1}{q} \right \rfloor ,$$
		so
		$$b_{I - 1}^{(K)} + 1 = IN + Kq + 1 \leq \pi(k) \leq IN + (K + 1)q.$$
		The following sketch of our argument motivates our definition of $y^{(K)}$. Let $f \in C(X)$. Then
		\begin{align*}
			& \frac{1}{\pi(k)} \sum_{j = 0}^{\pi(k) - 1} f_h \left( T_j y^{(K)} \right)	\\
			\approx	& \operatorname{Avg}_{\left[ a_0^{(K)}, b_0^{(K)}\right] \cup \left[ a_1^{(K)}, b_1^{(K)} \right] \cup \cdots \left[ a_{I - 1}^{(K)}, b_{I - 1}^{(K)} \right]} f_h \left( y^{(K)} \right) \\
			=	& \frac{1}{K p_0 + K p_1 + \cdots K p_{I - 1}} \sum_{i = 0}^{I - 1} \sum_{j = a_i^{(K)}}^{b_i^{(K)}} f_h \left( T_j y^{(K)} \right) \\
			=	& \frac{1}{K p_0 + K p_1 + \cdots K p_{I - 1}} \sum_{i = 0}^{I - 1} \left( b_i^{(K)} - a_i^{(K)} + 1 \right) \operatorname{Avg}_{\left[ a_i^{(K)} , b_i^{(K)} \right]} f_h \left( y^{(K)} \right) \\
			=	& \frac{1}{K q} \sum_{i = 0}^{I - 1} K p_i \operatorname{Avg}_{\left[ a_i^{(K)} , b_i^{(K)} \right]} f_h \left( y^{(K)} \right) \\
			=	& \sum_{i = 0}^{I - 1} \frac{p_i}{q} \operatorname{Avg}_{\left[ a_i^{(K)} , b_i^{(K)} \right]} f_h \left( y^{(K)} \right) \\
			\approx	& \sum_{i = 0}^{I - 1} \frac{p_i}{q} \operatorname{Avg}_{\left[ 0 , b_i^{(K)} - a_i^{(K)} \right]} f_h \left( x_i \right) \\
			\approx	& \sum_{i = 0}^{I - 1} \frac{p_i}{q} \int f_h \mathrm{d} \theta_i \\
			=	& \int f_h \mathrm{d} \nu ,
		\end{align*}
		where we write that $s(k) \approx t(k)$ if $|s(k) - t(k)| < \epsilon / 3$ for sufficiently large $k \in \mathbb{N}$. So it will suffice to verify those three claims.
		
			\textbf{Claim (i):} We first argue that
			$$\left| \left( \frac{1}{\pi(k)} \sum_{j = 0}^{\pi(k) - 1} f_h \left( T_j y^{(K)} \right) \right) - \operatorname{Avg}_{\left[ a_0^{(K)}, b_0^{(K)}\right] \cup \left[ a_1^{(K)}, b_1^{(K)} \right] \cup \cdots \left[ a_{I - 1}^{(K)}, b_{I - 1}^{(K)} \right]} f_h \left( y^{(K)} \right) \right| \leq \frac{\epsilon}{3}$$
			for sufficiently large $k \in \mathbb{N}$. We know that
			\begin{align*}
				& \frac{1}{\pi(k)} \sum_{j = 0}^{\pi(k) - 1} f_h \left( T_j y^{(K)} \right) \\
				=	& \frac{1}{\pi(k)} \left( \sum_{i = 0}^{I - 1} \sum_{j = a_i^{(K)}}^{b_i^{(K)}} f_h \left( T_j y^{(K)} \right) \right) + \frac{1}{\pi(k)}\left( \sum_{j = 0}^{a_0^{(K)}} f_h \left( T_j y^{(K)} \right) \right) \\
				& + \frac{1}{\pi(k)}\left( \sum_{i = 0}^{I - 2} \sum_{j = b_i^{(K)} + 1}^{a_{i + 1}^{(K)}} f_h \left( T_j y^{(K)} \right) \right) + \frac{1}{\pi(k)}\left( \sum_{j = b_{I - 1}^{(K)} + 1}^{\pi(k) - 1} f_h \left( T_j y^{(K)} \right) \right) \\
				=	& \frac{K q}{\pi(k)} \operatorname{Avg}_{\left[ a_0^{(K)}, b_0^{(K)}\right] \cup \left[ a_1^{(K)}, b_1^{(K)} \right] \cup \cdots \left[ a_{I - 1}^{(K)}, b_{I - 1}^{(K)} \right]} f_h \left( y^{(K)} \right) + \frac{1}{\pi(k)}\left( \sum_{j = 0}^{a_0^{(K)}} f_h \left( T_j y^{(K)} \right) \right) \\
				& + \frac{1}{\pi(k)}\left( \sum_{i = 0}^{I - 2} \sum_{j = b_i^{(K)} + 1}^{a_{i + 1}^{(K)}} f_h \left( T_j y^{(K)} \right) \right) + \frac{1}{\pi(k)}\left( \sum_{j = b_{I - 1}^{(K)} + 1}^{\pi(k) - 1} f_h \left( T_j y^{(K)} \right) \right) \\
				=	& \operatorname{Avg}_{\left[ a_0^{(K)}, b_0^{(K)}\right] \cup \left[ a_1^{(K)}, b_1^{(K)} \right] \cup \cdots \left[ a_{I - 1}^{(K)}, b_{I - 1}^{(K)} \right]} f_h \left( y^{(K)} \right) \\
				& + \frac{K q - \pi(k)}{\pi(k)} \operatorname{Avg}_{\left[ a_0^{(K)}, b_0^{(K)}\right] \cup \left[ a_1^{(K)}, b_1^{(K)} \right] \cup \cdots \left[ a_{I - 1}^{(K)}, b_{I - 1}^{(K)} \right]} f_h \left( y^{(K)} \right) \\
					&  + \frac{1}{\pi(k)}\left( \sum_{j = 0}^{a_0^{(K)}} f_h \left( T_j y^{(K)} \right) \right) + \frac{1}{\pi(k)}\left( \sum_{i = 0}^{I - 2} \sum_{j = b_i^{(K)} + 1}^{a_{i + 1}^{(K)}} f_h \left( T_j y^{(K)} \right) \right) \\
					& + \frac{1}{\pi(k)}\left( \sum_{j = b_{I - 1}^{(K)} + 1}^{\pi(k) - 1} f_h \left( T_j y^{(K)} \right) \right) .
			\end{align*}
			Therefore
			\begin{align*}
				& \left| \left( \frac{1}{\pi(k)} \sum_{j = 0}^{\pi(k) - 1} f_h \left( T_j y^{(K)} \right) \right) - \operatorname{Avg}_{\left[ a_0^{(K)}, b_0^{(K)}\right] \cup \left[ a_1^{(K)}, b_1^{(K)} \right] \cup \cdots \left[ a_{I - 1}^{(K)}, b_{I - 1}^{(K)} \right]} f_h \left( y^{(K)} \right) \right| \\
				\leq & \left| \frac{K q - \pi(k)}{\pi(k)} \operatorname{Avg}_{\left[ a_0^{(K)}, b_0^{(K)}\right] \cup \left[ a_1^{(K)}, b_1^{(K)} \right] \cup \cdots \left[ a_{I - 1}^{(K)}, b_{I - 1}^{(K)} \right]} f_h \left( y^{(K)} \right) \right| \\
				& + \left| \frac{1}{\pi(k)}\left( \sum_{j = 0}^{a_0^{(K)}} f_h \left( T_j y^{(K)} \right) \right) \right| \\
				& + \left| \frac{1}{\pi(k)}\left( \sum_{i = 0}^{I - 2} \sum_{j = b_i^{(K)} + 1}^{a_{i + 1}^{(K)}} f_h \left( T_j y^{(K)} \right) \right) \right| \\
					& + \left| \frac{1}{\pi(k)} \sum_{j = b_{I - 1}^{(K)} + 1}^{\pi(k) - 1} f_h \left( T_j y^{(K)} \right) \right| \\
				=	& \frac{\pi(k) - K q}{\pi(k)} \left| \operatorname{Avg}_{\left[ a_0^{(K)}, b_0^{(K)}\right] \cup \left[ a_1^{(K)}, b_1^{(K)} \right] \cup \cdots \left[ a_{I - 1}^{(K)}, b_{I - 1}^{(K)} \right]} f_h \left( y^{(K)} \right) \right| \\
				& + \left| \frac{1}{\pi(k)}\left( \sum_{j = 0}^{N} f_h \left( T_j y^{(K)} \right) \right) \right| \\
				& + \left| \frac{1}{\pi(k)}\left( \sum_{i = 0}^{I - 2} \sum_{j = b_i^{(K)} + 1}^{b_i^{(K)} + N} f_h \left( T_j y^{(K)} \right) \right) \right| + \left| \frac{1}{\pi(k)} \sum_{j = IN + K q + 1}^{\pi(k) - 1} f_h \left( T_j y^{(K)} \right) \right| \\
				\leq	& \frac{\pi(k) - K q}{\pi(k)} \left\| f_h \right\|_{C(X)} + \frac{N + 1}{\pi(k)} \left\|f_h\right\|_{C(X)} \\
				& + \frac{(I - 1) N}{\pi(k)} \left\| f_h \right\|_{C(X)} + \frac{\pi(k) - I N + K q + 1}{\pi(k)} \left\| f_h \right\|_{C(X)} \\
				\leq	& \left[ \frac{I N + 1}{\pi(k)} + \frac{N + 1}{\pi(k)} + \frac{(I - 1) N}{\pi(k)} + \frac{q}{\pi(k)} \right] \cdot \left\| f_h \right\|_{C(X)} \\
				\stackrel{k \to \infty}{\to}	& 0 .
			\end{align*}
			This establishes our estimate for large $k$.
			
			\textbf{Claim (ii):} We next argue that
			$$\left| \left( \sum_{i = 0}^{I - 1} \frac{p_i}{q} \operatorname{Avg}_{\left[ a_i^{(K)} , b_i^{(K)} \right]} f_h \left( y^{(K)} \right) \right) - \left( \sum_{i = 0}^{I - 1} \frac{p_i}{q} \operatorname{Avg}_{\left[ 0 , b_i^{(K)} - a_i^{(K)} \right]} f_h \left( x_i \right) \right) \right| < \frac{\epsilon}{3}$$
			for all $k \in \mathbb{N}$. To see this, we can note that
			\begin{align*}
				& \left| \left( \sum_{i = 0}^{I - 1} \frac{p_i}{q} \operatorname{Avg}_{\left[ a_i^{(K)} , b_i^{(K)} \right]} f_h \left( y^{(K)} \right) \right) - \left( \sum_{i = 0}^{I - 1} \frac{p_i}{q} \operatorname{Avg}_{\left[ 0 , b_i^{(K)} - a_i^{(K)} \right]} f_h \left( x_i \right) \right) \right| \\
				=	& \left| \sum_{i = 0}^{I - 1} \frac{p_i}{q} \; \frac{1}{b_i^{(K)} - a_i^{(K)} + 1} \sum_{j = 0}^{b_i^{(K)} - a_i^{(K)}} \left( f_h \left( T_{j + a_i^{(K)}} y^{(K)} \right) - f_h \left( T_j x_i \right) \right) \right| \\
				=	& \left| \sum_{i = 0}^{I - 1} \frac{p_i}{q} \; \frac{1}{K p_i} \sum_{j = 0}^{K p_i - 1} \left( f_h \left( T_{j + a_i^{(K)}} y^{(K)} \right) - f_h \left( T_j x_i \right) \right) \right| \\
				\leq	& \sum_{i = 0}^{I - 1} \frac{p_i}{q} \frac{1}{K p_i} \sum_{j = 0}^{K p_i - 1} \left| f_h \left( T_{j + a_i^{(K)}} y^{(K)} \right) - f_h \left( T_j x_i \right) \right| \\
				(\dagger) <	& \sum_{i = 0}^{I - 1} \frac{p_i}{q} \; \frac{1}{K p_i} \sum_{j = 0}^{K p_i - 1} \frac{\epsilon}{3} \\
				=	& \frac{\epsilon}{3} ,
			\end{align*}
			where the estimate $(\dagger)$ follows from the fact that $y$ is a $\delta$-tracing of $\xi^{(K)}$.
			
			\textbf{Claim (iii):} Our third step is to show that
			$$\left| \left( \sum_{i = 0}^{I - 1} \frac{p_i}{q} \operatorname{Avg}_{\left[ 0 , b_i^{(K)} - a_i^{(K)} \right]} f_h \left( x_i \right)
			\right) - \left( \sum_{i = 0}^{I - 1} \frac{p_i}{q} \int f_h \mathrm{d} \theta_i \right) \right| < \frac{\epsilon}{3}$$
			for sufficiently large $k \in \mathbb{N}$. This follows because
			\begin{align*}
				& \left| \left( \sum_{i = 0}^{I - 1} \frac{p_i}{q} \operatorname{Avg}_{\left[ 0 , b_i^{(K)} - a_i^{(K)} \right]} f_h \left( x_i \right)
				\right) - \left( \sum_{i = 0}^{I - 1} \frac{p_i}{q} \int f_h \mathrm{d} \theta_i \right) \right| \\
				=	& \left| \sum_{i = 0}^{I - 1} \frac{p_i}{q} \left( \operatorname{Avg}_{\left[ 0 , b_i^{(K)} - a_i^{(K)} \right]} f_h \left( x_i \right) - \int f_h \mathrm{d} \theta_i \right) \right| \\
				\leq	& \sum_{i = 0}^{I - 1} \frac{p_i}{q} \left| \operatorname{Avg}_{\left[ 0 , b_i^{(K)} - a_i^{(K)} \right]} f_h \left( x_i \right) - \int f_h \mathrm{d} \theta_i \right| \\
				= 	& \sum_{i = 0}^{I - 1} \frac{p_i}{q} \left| \left( \frac{1}{K p_i} \sum_{j = 0}^{K p_i - 1} f_h \left( T_j x_i \right) \right) - \int f_h \mathrm{d} \theta_i \right|
			\end{align*}
			If $k$ is sufficiently large that
			\begin{align*}
				\left| \left( \frac{1}{K p_i} \sum_{j = 0}^{K p_i - 1} f_h \left( T_j x_i \right) \right) - \int f_h \mathrm{d} \theta_i \right|	& < \frac{\epsilon}{3}	& \textrm{(for $i = 0 , 1, \ldots, I - 1$)} ,
			\end{align*}
			then
			$$\sum_{i = 0}^{I - 1} \frac{p_i}{q} \left| \left( \frac{1}{K p_i} \sum_{j = 0}^{K p_i - 1} f_h \left( T_j x_i \right) \right) - \int f_h \mathrm{d} \theta_i \right| < \sum_{i = 0}^{I - 1} \frac{p_i}{q} \; \frac{\epsilon}{3} = \frac{\epsilon}{3} .$$
		
		Taking these three claims together, we can say that
		\begin{align*}
			& \left| \left( \frac{1}{\pi(k)} \sum_{j = 0}^{\pi(k) - 1} f_h \left( T_j y^{(K)} \right) \right) - \int f \mathrm{d} \nu \right| \\
			\leq	& \left| \left( \frac{1}{\pi(k)} \sum_{j = 0}^{\pi(k) - 1} f_h \left( T_j y^{(K)} \right) \right) - \operatorname{Avg}_{\left[ a_0^{(K)}, b_0^{(K)}\right] \cup \left[ a_1^{(K)}, b_1^{(K)} \right] \cup \cdots \left[ a_{I - 1}^{(K)}, b_{I - 1}^{(K)} \right]} f_h \left( y^{(K)} \right) \right| \\
			& + \left| \left( \sum_{i = 0}^{I - 1} \frac{p_i}{q} \operatorname{Avg}_{\left[ a_i^{(K)} , b_i^{(K)} \right]} f_h \left( y^{(K)} \right) \right) - \left( \sum_{i = 0}^{I - 1} \frac{p_i}{q} \operatorname{Avg}_{\left[ 0 , b_i^{(K)} - a_i^{(K)} \right]} f_h \left( x_i \right) \right) \right| \\
			& + \left| \left( \sum_{i = 0}^{I - 1} \frac{p_i}{q} \operatorname{Avg}_{\left[ 0 , b_i^{(K)} - a_i^{(K)} \right]} f_h \left( x_i \right)
			\right) - \left( \sum_{i = 0}^{I - 1} \frac{p_i}{q} \int f_h \mathrm{d} \theta_i \right) \right| \\
			<	& \frac{\epsilon}{3} + \frac{\epsilon}{3} + \frac{\epsilon}{3} \\
			=	& \epsilon 
		\end{align*}
		for sufficiently large $k \in \mathbb{N}$.
		
		For each $h \in \{1, \ldots, H\}$, choose $k_h \in \mathbb{N}$ such that
		$$k \geq k_h \Rightarrow \left| \left( \frac{1}{\pi(k)} \sum_{j = 0}^{\pi(k) - 1} f_h \left( T_j y^{(K)} \right) \right) - \int f \mathrm{d} \nu \right| < \epsilon .$$
		Then if $k \geq \max \left\{ k_0, k_1 , \ldots, k_H \right\}$, it follows that $y^{(K)} \in E(\nu, \epsilon, H, k_0, \pi)$.
	\end{proof}
	
	\begin{proof}[Proof of Theorem \ref{Weak Spec Theorem with sampling}]
		We can metrize $\mathcal{M}(X)$ with the metric $\operatorname{dist} : \mathcal{M}(X) \times \mathcal{M}(X) \to [0, 1]$ defined by
		$$\operatorname{dist}(\beta_1, \beta_2) = \sum_{h = 1}^\infty \min 2^{-h} \left\{ \left| \int f_h \mathrm{d} (\beta_1 - \beta_2) \right| , 1 \right\} .$$
		For $\nu \in \mathcal{F}, k_0 \in \mathbb{N}, n \in \mathbb{N}, \pi \in \Pi$, write
		$$B(\nu, n, k_0) = \left\{ x \in X : \exists k \geq k_0 \; \left( \operatorname{dist}(\mu_{x, k}, \nu) < 1/n \right) \right\} .$$
		Choose $H_n \in \mathbb{N}$ such that $2^{-H_n} < 1/(2n)$. We claim that
		$$B(\nu, n, k_0, \pi) \supseteq E(\nu, 1/(2n), H_n, k_0) .$$
		If $x \in E(\nu, 1/(2n), H_n, k_0, \pi)$, then there exists $k \geq k_0$ such that
		\begin{align*}
			\operatorname{dist}(\mu_{x, \pi(k)}, \nu)	& = \sum_{h = 1}^\infty 2^{-h} \min \left\{ \int f_h \mathrm{d} \left( \mu_{x, \pi(j)} - \nu \right) , 1 \right\} \\
			& < 2^{-1} \frac{1}{2n} + 2^{-2} \frac{1}{2 n} + \cdots + 2^{-H_n} \frac{1}{2 n} + \sum_{h = H_\epsilon + 1}^\infty 2^{-h} \\
			& < \frac{1}{2n} + \frac{1}{2n} \\
			& = 1 / n.
		\end{align*}
		Thus $x \in B(\nu, n, k_0, \pi)$.
		
		We claim that $X ' \supseteq \bigcap_{\pi \in \Pi} \bigcap_{\nu \in \mathcal{F}} \bigcap_{n = 1}^\infty \bigcap_{k_0 = 1}^\infty B(\nu, n, k_0, \pi)$. Let \linebreak $x \in \bigcap_{\pi \in \Pi} \bigcap_{\nu \in \mathcal{F}} \bigcap_{n = 1}^\infty \bigcap_{k_0 = 1}^\infty B(\nu, n, k_0, \pi)$, and consider some $\nu \in \mathcal{M}_T(X)$. Choose a sequence $(\nu_\ell)_{\ell = 1}^\infty$ in $\mathcal{F}$ such that $\operatorname{dist}(\nu, \nu_\ell) < 1 / \ell$ for all $\ell \in \mathbb{N}$. Construct a sequence $(k_\ell)_{\ell = 1}^\infty$ in $\mathbb{N}$ recursively as follows:
		\begin{itemize}
			\item \textbf{Basis step:} Choose $k_1 \in \mathbb{N}$ such that $\operatorname{dist}\left( \mu_{x, \pi \left( k_1 \right)}, \nu_1 \right) < 1$, which exists because $x \in B(\nu_n, n, 1, \pi)$.
			
			\item \textbf{Recursive step:} Suppose we've chosen $k_1 < k_2 < \cdots < k_\ell$ such that $\operatorname{dist}\left(\mu_{x, \pi (k_n)} , \nu_n \right) < 1 / n$ for $n = 1, \ldots, \ell$. Chose $k_{\ell + 1} \geq k_\ell + 1$ such that $\operatorname{dist} \left( \mu_{x, \pi \left( k_{\ell + 1} \right)} , \nu_{\ell + 1} \right) < 1 / (\ell + 1)$, which exists because $x \in B(\nu_{\ell + 1}, \ell + 1, k_\ell + 1, \pi)$.
		\end{itemize}
		It follows then that
		$$\operatorname{dist}\left( \mu_{x, k_\ell} , \nu \right) \leq \operatorname{dist}\left( \mu_{x, k_\ell} , \nu_\ell \right) + \operatorname{dist} \left( \nu_\ell , \nu \right) < 2 / \ell \stackrel{\ell \to \infty}{\to} 0 ,$$
		i.e. $\nu \in \operatorname{LS} \left( \left( \mu_{x, k} \right)_{k = 1}^\infty \right)$.
		
		But $\bigcap_{\nu \in \mathcal{F}} \bigcap_{n = 1}^\infty \bigcap_{k_0 = 1}^\infty B(\nu, n, k_0)$ is a countable intersection of residual sets, and thus itself residual.
	\end{proof}
	
	\begin{Cor}
		Let $\mathbf{F} = \left( \{0, 1, \ldots, k - 1\} \right)_{k = 1}^\infty$, and suppose that $T : \mathbb{N}_0 \curvearrowright X$ is a H\"older action on $X$ that has the Very Weak Specification Property. Suppose $\Pi$ is a countable sampling family. Then the set of $x \in X$ such that $\operatorname{LS} \left( \left( \alpha_{B \left(x; r_{\pi(k)}\right)} \circ \operatorname{Avg}_{F_{\pi(k)}} \right)_{k = 1}^\infty \right) = \mathcal{M}_T(X)$ for all $(r_k)_{k = 1}^\infty$ that decay $(X, \rho, H, L, \mathbf{F})$-fast and $\pi \in \Pi$ is a residual subset of $X$.
	\end{Cor}
	
	\begin{proof}
		Lemma \ref{Singly local pointwise reduction} tells us that this is exactly the set considered in Theorem \ref{Weak Spec Theorem with sampling}.
	\end{proof}
	
	Our Theorem \ref{Weak Spec Theorem with sampling} strengthens the following result of J. Li and M. Wu, since the Specification Property implies the Very Weak Specification Property.
	
	\begin{Cor}\label{Li}\cite[Theorem 1.3]{LiWuOscillation}
		Suppose $T : \mathbb{N}_0 \curvearrowright X$ has the Specification Property, and let $f \in C_\mathbb{R}(X)$ be a real-valued continuous function on $X$. Then the set
		$$\left\{ x \in X : \liminf_{k \to \infty} \frac{1}{k} \sum_{j = 0}^{k - 1} f \left( T_j x \right) = \underline{a}(f) , \; \limsup_{k \to \infty} \frac{1}{k} \sum_{j = 0}^{k - 1} f \left( T_j x \right) = \overline{a}(f) \right\} $$
		is residual.
	\end{Cor}
	
	\begin{proof}
		Let $\Pi = \left\{ k \mapsto k \right\}$ be the sampling family consisting solely of the identity function $\mathbb{N} \to \mathbb{N}$, and consider $x \in X^\Pi$. Since the Specification Property implies the Very Weak Specification Property, Theorem \ref{Weak Spec Theorem with sampling} tells us that $X^\Pi$ is residual. Let $\theta_1, \theta_2 \in \partial_e \mathcal{M}_T(X)$ such that
		\begin{align*}
			\int f \mathrm{d} \theta_1	& = \underline{a}(f) , \\
			\int f \mathrm{d} \theta_2	& = \overline{a}(f) .
		\end{align*}
		Then there exist $k_1^{(i)} < k_2^{(i)} < k_3^{(i)} \cdots$ for $i = 1, 2$ such that $\lim_{\ell \to \infty} \mu_{x, k_\ell^{(i)}} = \theta_i$. Thus
		\begin{align*}
			\underline{a}(f)	& \leq \liminf_{k \to \infty} \frac{1}{k} \sum_{j = 0}^{k - 1} f \left( T_j x \right)	& \leq \lim_{\ell \to \infty} \frac{1}{k_\ell^{(1)}} \sum_{j = 0}^{k_\ell^{(1)} - 1} f \left( T_j x \right)	& = \underline{a}(f) \\
			&	& \Rightarrow \liminf_{k \to \infty} \frac{1}{k} \sum_{j = 0}^{k - 1} f \left( T_j x \right)	& = \underline{a}(f) , \\
			\overline{a}(f)	& \geq \limsup_{k \to \infty} \frac{1}{k} \sum_{j = 0}^{k - 1} f \left( T_j x \right)	& \geq \lim_{\ell \to \infty} \frac{1}{k_\ell^{(2)}} \sum_{j = 0}^{k_\ell^{(2)} - 1} f \left( T_j x \right)	& = \overline{a}(f) \\
			&	& \Rightarrow \limsup_{k \to \infty} \frac{1}{k} \sum_{j = 0}^{k - 1} f \left( T_j x \right)	& = \overline{a}(f) .
		\end{align*}
		Therefore
		$$X^\Pi \subseteq \left\{ x \in X : \liminf_{k \to \infty} \frac{1}{k} \sum_{j = 0}^{k - 1} f \left( T_j x \right) = \underline{a}(f) , \limsup_{k \to \infty} \frac{1}{k} \sum_{j = 0}^{k - 1} f \left( T_j x \right) = \overline{a}(f) \right\},$$
		meaning the latter is residual.
	\end{proof}

\section*{Acknowledgments}

This paper is written as part of the author's graduate studies. He is grateful to his beneficent advisor, professor Idris Assani, for no shortage of helpful guidance.

An earlier version of this paper referred to ``tempero-spatial differentiations." Professor Mark Williams pointed out that the more correct portmanteau would be ``temporo-spatial." We thank Professor Williams for this observation.

\bibliography{Bibliography}
\end{document}